\documentclass[preprint,12pt]{hbarxiv}
\usepackage{color}
\usepackage{soul}
\usepackage{amssymb}
\usepackage{amsthm}
\usepackage{amsmath}
\usepackage{graphics}
\usepackage{csquotes}
\newtheorem{theorem}{Theorem}[section]
\newtheorem{lemma}[theorem]{Lemma}
\newtheorem{corollary}[theorem]{Corollary}
\newtheorem{definition}[theorem]{Definition}
\begin{document}
\begin{frontmatter}
\title{The Generalized Fractional Benjamin-Bona-Mahony Equation: Analytical and Numerical  Results}
\author[itu]{Goksu Oruc \corref{cor1}}
\ead{topkarci@itu.edu.tr}
\author[ozu]{Handan Borluk}
\ead{handan.borluk@ozyegin.edu.tr}
\author[itu]{Gulcin M. Muslu  }
\cortext[cor1]{Corresponding author}
\ead{gulcin@itu.edu.tr}

\address[itu]{Istanbul Technical University, Department of Mathematics, Maslak,
         Istanbul,  Turkey.}
 \address[ozu]{Ozyegin University, Department of Natural and Mathematical Sciences, Cekmekoy, Istanbul,  Turkey.}

\begin{abstract}
The generalized fractional Benjamin-Bona-Mahony (gfBBM) equation \mbox{models} the propagation of small amplitude long unidirectional waves in a nonlocally and nonlinearly elastic medium. The equation involves two fractional terms unlike the well-known fBBM equation. In this paper, we prove  local existence and uniqueness of the solutions for the Cauchy problem by using energy method. The sufficient conditions for the existence of solitary wave solutions are obtained. The Petviashvili method is proposed for the generation of the solitary wave solutions and their evolution in time is investigated numerically by Fourier spectral method. The efficiency of the numerical methods is tested and the relation between nonlinearity and fractional dispersion is observed by various numerical experiments.
\end{abstract}

\begin{keyword}
Generalized Fractional Benjamin-Bona-Mahony equation, Conserved Quantities, Local Existence, Solitary Waves, Petviashvili Method
\end{keyword}

\end{frontmatter}
\renewcommand{\theequation}{\arabic{section}.\arabic{equation}}
\setcounter{equation}{0}
\section{Introduction}
This paper is concerned with the generalized fractional Benjamin-Bona-Mahony (gfBBM) equation
\begin{equation} \label{gfBBM}
   u_t+  u_x + \frac{1}{2}(u^{p+1})_x+   \frac{3}{4}D^{\alpha} u_{x}+ \frac{5}{4}D^{\alpha} u_{t}=0,
  \end{equation}
which  models the propagation of small amplitude long unidirectional waves in a nonlocally and nonlinearly elastic medium. The equation is derived in \cite{erbayelastic,hae} as a special case of the generalized fractional Camassa-Holm equation by using an asymptotic expansion technique. 
Here $p$ is a positive integer and the operator $D^{\alpha} =(- \Delta)^{\frac{\alpha}{2}}$ denotes the Riesz potential of order $-\alpha$, for any $\alpha \in \mathbb{R}$.
The operator 
can be  defined via Fourier transform by
\begin{equation*}
  \widehat{D^{\alpha}q}(\xi)=|\xi|^{\alpha}\hat{q}(\xi),
\end{equation*}
where $\hat{q}$ is the Fourier transform of a function $q$.

The effects of the relation between the nonlinearity and the dispersion on the dynamics of solutions has been the focus of many studies. The problem mostly handled  by fixing the dispersion and increasing the nonlinearity. Studies for the more physically significant case with lower dispersion has
become popular only in the recent years. The well-known Korteweg-de Vries (KdV)  and Benjamin-Bona-Mohany (BBM) equations are investigated by using corresponding fractional forms such as
\begin{equation*} \label{fkdv}
  u_t+  u_x + u^p u_x+ D^{\alpha}u_x=0
\end{equation*}
fractional KdV (fKdV) equation and
\begin{equation*} \label{fbbm}
  u_t+  u_x + u^pu_x+ D^{\alpha}u_t=0
\end{equation*}
fractional BBM (fBBM) equation. These equations have been intensively studied in the past years for $\alpha \geq 1$ in terms of global well-posedness, stability and blow-up etc. We refer \cite{fonseca, albert, pava} for a more detailed discussion and review. The research on the more delicate case $\alpha\in (0, \,1)$  has been increased in the last few years.
For $p=1$, Linares et. al. \cite{linares} proved the local well-posedness for the
Cauchy problem  and they have also investigated the solitary wave solutions in terms of existence and stability in \cite{linares2015}. The stability and linear instability results for a general nonlinearity are obtained by Pava \cite{pava}. In \cite{kleinfBBM}  the blow-up and the global existence problems are handled and  solitary wave solutions for the fKdV are  constructed numerically.  Duran used efficient numerical methods to investigate the solitary wave solution of the  fKdV equation in \cite{aduran}.

The gfBBM equation contains the fractional terms of both  the fKdV and the fBBM equations, but unlike them, the gfBBM equation models a physical phenomena. The effects of these terms on the  solutions when occurring together, such as well-posedness of the Cauchy problem, existence of solitary waves and the nature of solutions in time is therefore a curious problem. The aim of the current study is to investigate the dynamics of the gfBBM equation with a general power type nonlinear term.

The paper is organized as follows: In Section 2, we prove the local existence and uniqueness of the Cauchy problem for the gfBBM equation together with the initial condition
\begin{equation}
u(x,0)=\phi (x) \label{incon}
\end{equation}
by using the energy method. In Section 3,   we derive the conserved quantities for the gfBBM equation.
The existence-nonexistence results for solitary wave solutions are given in Section 4. We use Pohozaev type identites to show the non-existence of solitary wave solutions and the results of \cite{franklenzmann} are applied for the existence of positive solitary waves for certain values of  $\alpha,~p$ and the wave speed $c$.
Section 5 is devoted to numerical investigation of the solutions.  We  construct the  solitary wave solutions numerically by using Petviashvili method. For the time evolution of the constructed solution we  propose a numerical scheme combining a
Fourier pseudo-spectral method for space and a fourth order Runge-Kutta method for the time integration. We perform numerical experiments for several values of $\alpha$ and $p$ to investigate the effects of dispersion and nonlinearity.

\noindent
 Throughout this study, $L^p(\mathbb{R})$ is the usual Lebesgue space with the norm $\parallel . \parallel_{L^p}$ for $1\leq p \leq \infty$. $H^s(\mathbb{R})$ is the Sobolev space with the norm
\begin{equation*}
  \| u \|_{H^{s}} = (\int_{\mathbb{{R}}} (1+|k|^2)^s |\hat{u}(k)|^2 dk )^{1/2}
\end{equation*}
for $s \in \mathbb{{R}}$ and  $C$ denotes the generic constant. Here, the Fourier transform and its inverse for the given function $u\in L^2(\mathbb{R})$    defined as
\begin{equation}
u(x)=\frac{1}{2\pi} \int_{\mathbb{{R}}}\hat{u} (k)e^{ik x} dk, \hspace*{30pt} {\hat{u}(k)}=\int_{\mathbb{{R}}} u(x) e^{-ik x} dx.
\end{equation}
\noindent
$J^s=(I-\Delta)^{\frac{s}{2}}$ denotes the Bessel potentials of order $-s$ with
\mbox{$\| J^s u \|_{L^{2}}=   \| u \|_{H^{s}}$}.

\setcounter{equation}{0}
\section{Local Existence and Uniqueness}
\setcounter{equation}{0}

The local well-posedness of solutions in Sobolev spaces  for the fractional Benjamin-Bona-Mahony equation with  the quadratic nonlinearity
\begin{equation*} \label{fbbm}
  u_t+  u_x + u u_x+ D^{\alpha}u_t=0
\end{equation*}
is studied by using energy estimates in \cite{linares}. As in stated \cite{he}, this method does not provide the uniqueness since one of the terms can not be controlled by the appropriate Sobolev norm.
The same problem is also observed for the gfBBM equation. Therefore, we follow the idea given in \cite{he}. To prove the local existence and uniqueness of the Cauchy problem, we consider the following regularization
 \begin{eqnarray}
 && u_t^{\epsilon} + u_x^{\epsilon} + \frac{1}{2} [(u^{\epsilon})^{p+1}]_x + \frac{3}{4} D ^{\alpha} u_x^{\epsilon}+ \frac{5}{4} D ^{\alpha}u_t^{\epsilon} - \epsilon u^{\epsilon}_{xx} =0,  \label{reg1}  \\
 &&u^\epsilon (x,0)=u_0^\epsilon(x).  \label{regin}
\end{eqnarray}
The eq.  \eqref{reg1} is rewritten as
\begin{equation}
   u_t^{\epsilon} +  \big(\frac{I+ \displaystyle\frac{3}{4}D^{\alpha}}{I+ \displaystyle\frac{5}{4}D^{\alpha}} \big)  u_x^{\epsilon}  -  \big( \frac{1}{I+ \displaystyle\frac{5}{4}D^{\alpha}} \big)  \epsilon u^{\epsilon}_{xx}  =  - \frac{1}{2 (I+ \displaystyle\frac{5}{4}D^{\alpha})} [(u^{\epsilon})^{p+1}]_x.
\end{equation}
The Duhamel formula implies that $u^\epsilon$  is the solution of the Cauchy problem \eqref{reg1}-\eqref{regin} if and only if $u^\epsilon$ is the solution of the integral equation $u^\epsilon=\Phi^{\epsilon}u^\epsilon $  where

\begin{equation}\label{reg2}
  (\Phi^{\epsilon} u^{\epsilon})(x,t) = S_t u_0^\epsilon(x) - \frac{1}{2} \int_{0}^{t} S_{t-\tau} \big[ \frac{\partial_x}{(I+\frac{5}{4}D^{\alpha})} (u^{\epsilon})^{p+1} \big] (x,\tau) d\tau,
\end{equation}
with
\begin{equation*}
  S_t u= \mathcal{F}^{-1}\bigg( e^{-\big[\frac{1+ \frac{3}{4}|\xi|^{\alpha}}{1+ \frac{5}{4}|\xi|^{\alpha}} i \xi + \frac{\epsilon \xi^2}{1+ \frac{5}{4}|\xi|^{\alpha}} \big ] t} \bigg)* u(x).
\end{equation*}
Here, the symbol * denotes the convolution operation. We need the following lemmas in order to proceed with the fixed point theorem:

\begin{lemma}\label{lemmareg1}
  Let $0<\alpha<1$ and $r\geq 0$. We have
  \begin{equation}\label{lemmaeq1}
    \| \frac{\partial_x}{(I+\displaystyle\frac{5}{4}D^{\alpha})} S_t(uv) \|_{H^r}
     \leq C(\epsilon,t) \| u\|_{H^r} \|v\|_{H^r}
  \end{equation}
  where
  \begin{equation}
S_t u= \mathcal{F}^{-1}\bigg( e^{-\big[ \frac{1+ \frac{3}{4}|\xi|^{\alpha}}{1+ \frac{5}{4}|\xi|^{\alpha}} i \xi + \frac{\epsilon \xi^2}{1+ \frac{5}{4}|\xi|^{\alpha}} \big ] t} \bigg)* u(x), \hspace*{20pt}
C(\epsilon,t)= C (\epsilon t)^{- \frac{\frac{3}{2}- \alpha}{2-\alpha}}. \label{oper}
\end{equation}

\end{lemma}
\begin{proof}
By duality, proving the lemma is equivalent to proving
  \begin{equation*}
     \int_{\mathbb{R}} \frac{\partial_x}{(I+\frac{5}{4}D^{\alpha})} S_t(u(x)v(x))\overline{w}(x) dx
     \leq C(\epsilon, t) \| u\|_{H^r} \|v\|_{H^r}\|w\|_{H^{-r}}
  \end{equation*}
 for all \mbox{$w \in {\cal S}(\mathbb{R})$}. Thanks to the Plancherel identity, one gets
  \begin{eqnarray*}
     && \int_{\mathbb{R}} \frac{\partial_x}{(I+\frac{5}{4}D^{\alpha})} S_t(u(x)v(x))\overline{w}(x) dx \nonumber \\
    &&\hspace{80pt}= \int_{\mathbb{R}} \frac{i \xi}{1+\frac{5}{4}|\xi|^{\alpha}}\bigg[ e^{-\big(\frac{1+ \frac{3}{4}|\xi|^{\alpha}}{1+ \frac{5}{4}|\xi|^{\alpha}} i \xi + \frac{\epsilon \xi^2}{1+ \frac{5}{4}|\xi|^{\alpha}} \big ) t} \bigg] (\widehat{u}*\widehat{v})(\xi)\widehat{\overline{w}}(\xi) d \xi\nonumber \\
     && \hspace{80pt}=\int_{\mathbb{R}} \int_{\mathbb{R}} \frac{i \xi}{1+\frac{5}{4}|\xi|^{\alpha}}\bigg[ e^{-\big(\frac{1+ \frac{3}{4}|\xi|^{\alpha}}{1+ \frac{5}{4}|\xi|^{\alpha}} i \xi + \frac{\epsilon \xi^2}{1+ \frac{5}{4}|\xi|^{\alpha}} \big ) t} \bigg]\widehat{u}(\xi-\eta)\widehat{v}(\eta)\widehat{\overline{w}}(\xi) d \eta d \xi.
  \end{eqnarray*}
Let us define
\begin{equation}
  \widehat{U} = \langle \xi \rangle ^r \widehat{u}, \hspace*{20pt}
  \widehat{V} = \langle \xi \rangle^r \widehat{v},   \hspace*{20pt}
  \widehat{W} = \langle \xi \rangle^{-r} \widehat{w}, \label{def}
\end{equation}
where $\langle \xi \rangle = (1+\xi^2)^{1/2}$. By using the triangle inequality  $\langle \xi \rangle ^r  \leq C \langle \xi-\eta \rangle ^r  \langle \eta \rangle^r$ and
$\frac{ |\xi|}{1+ \frac{5}{4}|\xi|^{\alpha}} \leq  \langle \xi \rangle ^{1-\alpha}$, we have
\begin{eqnarray*}
&& \mid \int_{\mathbb{R}} \frac{\partial_x}{(I+\frac{5}{4}D^{\alpha})} S_t(u(x)v(x))\overline{w}(x) dx \mid  \nonumber \\
&& \hspace{20pt} =   \bigg | \int_{\mathbb{R}} \int_{\mathbb{R}}    \frac{i \xi}{1+\frac{5}{4}|\xi|^{\alpha}} e^{-\big[\frac{1+ \frac{3}{4}|\xi|^{\alpha}}{1+ \frac{5}{4}|\xi|^{\alpha}} i \xi + \frac{\epsilon \xi^2}{1+ \frac{5}{4}|\xi|^{\alpha}} \big ] t} \frac{\langle \xi \rangle^r }{\langle \xi-\eta \rangle^r \langle \eta \rangle^r } \widehat{U}(\xi-\eta)\widehat{V}(\eta)\widehat{\overline{W}}(\xi)   d \eta d \xi  \bigg | \nonumber \\
&& \hspace{20pt} \leq C \|W\|_{L^2} ~~ \| \frac{i \xi}{1+\frac{5}{4}|\xi|^{\alpha}} e^{-\big[\frac{1+ \frac{3}{4}|\xi|^{\alpha}}{1+ \frac{5}{4}|\xi|^{\alpha}} i \xi + \frac{\epsilon \xi^2}{1+ \frac{5}{4}|\xi|^{\alpha}} \big ] t} \widehat{U}*\widehat{V}  \|_{L^2} \nonumber \\
&& \hspace{20pt} \leq C\|W\|_{L^2}\|\widehat{U}*\widehat{V}\|_{L^{\infty}} \| \langle \xi \rangle^{1-\alpha} e^{ -\frac{\epsilon \xi^2}{1+ \frac{5}{4}|\xi|^{\alpha}}t}  \|_{L^2}
\nonumber \\
&& \hspace{20pt} \leq C\|W\|_{L^2} \|U\|_{L^2} \|V\|_{L^2} \| \langle \xi \rangle^{1-\alpha} e^{ -\frac{\epsilon \xi^2}{1+ \frac{5}{4}|\xi|^{\alpha}}t}  \|_{L^2}
\end{eqnarray*}
The inequality $e^{-x} \leq \displaystyle \frac{1}{x^\theta}$  for $x\in \mathbb{R}$  and $\theta \in \mathbb{R}$ yields that

\begin{eqnarray*}
 && \hspace*{-40pt}\mid \int_{\mathbb{R}} \frac{\partial_x}{(I+\frac{5}{4}D^{\alpha})} S_t(u(x)v(x))\overline{w}(x) dx \mid \nonumber  \\
 &&  \hspace*{40pt} \leq  C \|W\|_{L^2}  \|U\|_{L^2} \|V\|_{L^2}   \| \langle \xi \rangle^{1-\alpha} \frac{1}{ \big( \frac{\epsilon \xi^2}{1+ \frac{5}{4}|\xi|^{\alpha}}t  \big) ^{\theta}}  \|_{L^2} \nonumber \\
  &&  \hspace*{40pt}  \leq  C \|W\|_{L^2} \|U\|_{L^2} \|V\|_{L^2} \frac{1}{\big( \epsilon  t \big) ^{\theta}}  \|\frac{1}{\langle \xi \rangle^{\theta(2-\alpha)-(1-\alpha)}} \|_{L^2}
\nonumber \\
 && \hspace*{40pt} \leq C(\epsilon,t) \|W\|_{L^2}  \|U\|_{L^2}\|V\|_{L^2}
\end{eqnarray*}
 if $ \theta > \frac{\frac{3}{2}- \alpha}{2-\alpha}.$
 By using the definitions \eqref{def}, we obtain
 \begin{equation*}
 \mid \int_{\mathbb{R}} \frac{\partial_x}{(I+\frac{5}{4}D^{\alpha})} S_t(u(x)v(x))\overline{w}(x) dx \mid
     \leq  C(\epsilon,t) \|w\|_{H^{-r}} \|u\|_{H^r}  \|v\|_{H^r}.
 \end{equation*}
 \end{proof}

\begin{lemma} \label{lemproduct}
For $r > n/2$,  $H^r (\mathbb{R}^n)$ is an algebra with respect to the
product of functions. That is, if $u, v \in H^r(\mathbb{R}^n)$ then   $uv \in H^r(\mathbb{R}^n)$  and
\begin{equation}
\|uv \|_ {H^r (\mathbb{R}^n)}  \leq C \|u\|_ {H^r (\mathbb{R}^n)} \| v \|_ {H^r (\mathbb{R}^n)}.
\end{equation}
\end{lemma}

\begin{corollary} \label{cor}
 Let $0<\alpha<1$ and $r > 1/2$. We have
  \begin{equation}\label{lemmaeq1}
    \| \frac{\partial_x}{(I+\displaystyle\frac{5}{4}D^{\alpha})} S_t(u^p) \|_{H^r}
     \leq C(\epsilon,t) \| u\|^p_{H^r},
  \end{equation}
  where the operator $S_t$ and $C(\epsilon,t)$  are  given in the eq. \eqref{oper}.
\end{corollary}

\begin{lemma}\cite{runst} \label{lemma2}
Assume that $f\in C^k(\mathbb{R})$, $u,v\in H^r(\Omega)\cap L^\infty (\Omega)$ and $k=[r]+1$, where $r\geq 0$. Then we have
\begin{equation}
\|f(u)-f(v)\|_r \leq K(M) \|u-v\|_r
\end{equation}
if $\|u\|_\infty \leq M, \|v\|_\infty \leq M, \|u\|_r \leq M$  and $\|v\|_r \leq M$, where $K(M)$ is a constant depending
on $M$ and $s$.
\end{lemma}

Now, we  prove the existence and uniqueness of the local solution for the reqularized  problem  \eqref{reg1}-\eqref{regin} by using the contraction mapping principle.
\begin{theorem}
Assume that $u_0^\epsilon \in H^r (\mathbb{R})$ with $r>1/2$ and $0<\alpha<1$. Then, for any $\epsilon >0$ there exists a unique solution $u^\epsilon \in C([0,T_{\epsilon}],H^r(\mathbb{R}))$ to the Cauchy problem
\eqref{reg1}-\eqref{regin}, where $T_\epsilon=T(\|u_0^{\epsilon}\|,\epsilon)$.

\end{theorem}
\begin{proof}
Let $\overline{B}_T$ be the closed ball
\begin{equation*}
  \overline{B}_T=\{u \in C([0,T], H^r(\mathbb{R})), \|u\|_{L^{\infty}([0,T],H^r(\mathbb{R}))} \leq 2 \|u_0\|_{H^r(\mathbb{R})}\}.
\end{equation*}
We first prove  $\Phi^{\epsilon}$ maps $\overline{B}_T$ into $\overline{B}_T$ for $T$ small enough. From \eqref{reg2}, one gets
\begin{eqnarray*}\label{reg3}
  \|\Phi^{\epsilon} u^{\epsilon}(t)\|_{H^r} & = & \|S_tu_0^{\epsilon} - \frac{1}{2} \int_{0}^{t} S_{t-\tau} \big[ \frac{\partial_x}{(I+\frac{5}{4}D^{\alpha})} (u^{\epsilon})^{p+1} \big](\tau) d\tau \|_{H^r} \nonumber \\
  &\leq &  \|u_0^{\epsilon}\|_{H^r} +\frac{1}{2}  \int_{0}^{t} \| S_{t-\tau} \big[ \frac{\partial_x}{(I+\frac{5}{4}D^{\alpha})} (u^{\epsilon})^{p+1}(\tau) \big] \|_{H^r}  d\tau.
\end{eqnarray*}

\noindent
The Corollary \ref{cor} implies that
\begin{equation}\label{reg4}
  \| S_{t-\tau} \big[ \frac{\partial_x}{(I+\frac{5}{4}D^{\alpha})} (u^{\epsilon})^{p+1}(\tau)\big] \|_{H^r} \leq  C(\epsilon(t-\tau))^{-\frac{\frac{3}{2}-\alpha}{2-\alpha}}\|u^{\epsilon}(t)\|_{H^r}^{p+1}.
\end{equation}
Using \eqref{reg4}, we have
\begin{equation*}
  \|\Phi^{\epsilon} u^{\epsilon}(t)\|_{H^r} \leq \|u_0^{\epsilon}\|_{H^r} + C_{\epsilon,r}T \underset{t \in [0,T] }{\mbox{sup}} \|u^{\epsilon}(t)\|_{H^r}^{p+1}
\end{equation*}
By  choosing $T$ small enough to satisfy  $T \leq \displaystyle\frac{1}{C_{\epsilon,r} 2^{p+1}\|u_0^{\epsilon}\|_{H^r}^p}$ gives  that $\Phi^{\epsilon}$ maps $\overline{B}_T$ into $\overline{B}_T$.

\noindent
The next step is to prove that  $\Phi^{\epsilon}$ is contractive. Let $u_1^{\epsilon}$, $u_2^{\epsilon}$ $\in$ $\overline{B}_T$. The Duhamel formula  \eqref{reg2} gives
\begin{equation}
  \|\Phi^{\epsilon} u_1^{\epsilon}(t)- \Phi^{\epsilon} u_2^{\epsilon}(t)\|_{H^r}    \leq  \frac{1}{2} \int_{0}^{t} \| S_{t-\tau} \frac{\partial_x}{(I+\frac{5}{4}D^{\alpha})} \big[(u_1^{\epsilon})^{p+1}-(u_2^{\epsilon})^{p+1} \big](\tau) \|_{H^r} d\tau
\end{equation}
and  Corollary \ref{cor} and Lemma \ref{lemma2} imply
\begin{eqnarray}
\|\Phi^{\epsilon} u_1^{\epsilon}(t)- \Phi^{\epsilon} u_2^{\epsilon}(t)\|_{H^r}   &\leq &  C \int_{0}^{t}  (\epsilon(t-\tau))^{-\frac{\frac{3}{2}-\alpha}{2-\alpha}}  \| (u_1^{\epsilon})^{p+1}(\tau)- ({u_2^{\epsilon}})^{p+1}(\tau) \|_{H^r} d\tau ,  \nonumber\\
  &\leq & C_{\epsilon,r}T K(M) \underset{t \in [0,T] }{\mbox{sup}} \|u_1^{\epsilon}(t)-u_2^{\epsilon}(t)\|_{H^r}.  \nonumber
\end{eqnarray}
where $\|u_1^{\epsilon}\|_{H^r} \leq M $ and $\|u_2^{\epsilon} \|_{H^r} \leq M$. If we choose $T <\frac{1}{C_{\epsilon,r}K(M)}$, then  $\Phi^{\epsilon}$ is strictly contractive.

\noindent
To obtain the continuity with respect to the initial data, we consider the solutions $u^{\epsilon}$ and $v^{\epsilon}$ in $H^r(\mathbb{R})$ corresponding to initial conditions $u_0^{\epsilon}$ and  $v_0^{\epsilon}$, respectively,
with
$\|u_0^{\epsilon}\|_{H^r} \leq M $ and $\|v_0^{\epsilon} \|_{H^r} \leq M$. Similar computations show that
\begin{eqnarray}
&& \hspace*{-20pt}  \| u^{\epsilon}(t)- v^{\epsilon}(t)\|_{H^r}   \nonumber\\
  && \leq  \| S_t \big(u_0^{\epsilon}-v_0^{\epsilon} \big) \|_{H^r} + \int_{0}^{t} \| S_{t-\tau} \frac{\partial_x}{(I+\frac{5}{4}D^{\alpha})} \big((u^{\epsilon})^{p+1}-(v^{\epsilon})^{p+1} \big) (\tau)\|_{H^r} d\tau,  \nonumber \\
  && \leq  \| u_0^{\epsilon}- v_0^{\epsilon}\|_{H^r} + C_{\epsilon,r} T K(M) \sup_{t\in [0,T]}\| u^{\epsilon}(t)- v^{\epsilon}(t) \|_{H^r}.  \nonumber
\end{eqnarray}
The inequality
\begin{equation}
  \sup_{t\in [0,T]} \| u^{\epsilon}(t)- v^{\epsilon}(t)\|_{H^r}  \leq  \frac{1}{1-C_{\epsilon,r}TK(M)}\|u_0^{\epsilon}- v_0^{\epsilon}\|_{H^r} \nonumber
\end{equation}
with $1-C_{\epsilon,r}T K(M)>0$  yields that  the solution depends continuously on the given initial data since it is bounded by a
continuous function related to the difference of the initial data.
\end{proof}

\begin{theorem} \label{Thm1}

Let $0 < \alpha < 1$ and  the initial data in $u_0^{\epsilon} \in H^{r}(\mathbb{R})$ with  \mbox{$r >\frac{3}{2}-\alpha$}. Then the unique regularized
solution $C([0,T],H^{r}(\mathbb{R}))$   to the eqs. \eqref{reg1}-\eqref{regin}  satisfies  the energy estimate
\begin{equation}\label{ener}
  \frac{d}{dt} \| u^{\epsilon}(t)\|^{2}_{H^{r}(\mathbb{R})} \leq C \| u^{\epsilon}(t)\|^{p+2}_{H^{r}(\mathbb{R})}.
\end{equation}

\end{theorem}
\textbf{Proof:} Let $r=s+\frac{\alpha}{2}$. Applying the operator $J^s$ to  the eq. \eqref{reg1},  multiplying both sides of the equation by $J^s u^{\epsilon}$ and then integrating on the whole line, we have
\begin{equation}\label{wp1}
  \frac{1}{2} \frac{d}{dt} \int_{\mathbb{R}} ( |J^s u^{\epsilon}|^2 + \frac{5}{4} |J^{s+\frac{\alpha}{2}}u^{\epsilon}|^2 ) dx +\epsilon  \frac{d}{dt} \int_{\mathbb{R}} (J^s u^{\epsilon}_x)^2 dx = -\frac{p+1}{2} \int_{\mathbb{R}} J^s u^{\epsilon} J^s\big((u^{\epsilon})^p u^{\epsilon}_x \big )dx.
\end{equation}
Using the fractional Leibniz rule  \cite{kenig},  the term $ J^s\big((u^{\epsilon})^p u^{\epsilon}_x \big )$  is written as
\begin{equation*}
  J^s\big((u^{\epsilon})^p u^{\epsilon}_x \big ) =(u^{\epsilon})^pJ^su^{\epsilon}_x+ u^{\epsilon}_xJ^s(u^{\epsilon})^p+R,
\end{equation*}
where $R$ is the remainder. Then, the eq. \eqref{wp1} becomes
\begin{eqnarray} \label{wp2}
  \frac{1}{2}\frac{d}{dt} \int_{\mathbb{R}} ( |J^su^{\epsilon}|^2 &+& \frac{5}{4} |J^{s+\frac{\alpha}{2}}u^{\epsilon}|^2 )  dx   +\epsilon  \frac{d}{dt} \int_{\mathbb{R}} (J^s u^{\epsilon}_x)^2 dx  \nonumber \\
  &=& -\frac{p+1}{2} \left( \int_{\mathbb{R}}(u^{\epsilon})^pJ^su^{\epsilon}_xJ^s u^{\epsilon} dx + \int_{\mathbb{R}} u^{\epsilon}_xJ^s(u^{\epsilon})^pJ^s u^{\epsilon} dx + \int_{\mathbb{R}}RJ^su^{\epsilon} dx \right).  \nonumber \\
&&
\end{eqnarray}
Using the  integration by parts  for the first term of RHS, we have
\begin{equation*}
  \int_{\mathbb{R}}  (u^{\epsilon})^p J^s u^{\epsilon}_x J^s u^{\epsilon} dx=-\frac{1}{2} \int_{\mathbb{R}} [(u^{\epsilon})^{p}]_x(J^su^{\epsilon})^2dx.
\end{equation*}
Thus by using the H{\"o}lder's inequality,  the first term of RHS is estimated as
\begin{eqnarray} \label{local1}
  \int_{\mathbb{R}} |[(u^{\epsilon})^p]_x(J^su^{\epsilon})^2|dx &\leq& C \|[(u^{\epsilon})^p]_x\|_{L^{q_1}}   \|J^s u^{\epsilon}\|^2_{L^{q_2}} \nonumber\\
  &\leq& C \|[(u^{\epsilon})^p]_x\|_{H^{s+\frac{\alpha}{2}-1}}  \|J^s u^{\epsilon}\|^2_{H^{\frac{\alpha}{2}}} \nonumber \\
  &\leq & C \|u^{\epsilon}\|^{p}_{H^{s+\frac{\alpha}{2}}}   \|J^s u^{\epsilon}\|^{2}_{H^{\frac{\alpha}{2}}}  \nonumber \\
  &\leq & C \|u^{\epsilon}\|^{p+2}_{H^{s+\frac{\alpha}{2}}}.
\end{eqnarray}
Here, we have used the following Sobolev imbeddings
\begin{eqnarray}
  H^{s+\frac{\alpha}{2}-1} & \hookrightarrow & L^{q_1},~~\mbox{for}~~{q_1}\leq \frac{ 2}{3-2s-\alpha}, \nonumber\\
  H^{\frac{\alpha}{2}} & \hookrightarrow & L^{q_2},~~\mbox{for}~~{q_2}\leq \frac{2}{1-\alpha},   \label{embedding1}
\end{eqnarray}
where $\frac{1}{q_1} + \frac{2}{q_2} \leq 1$
 implies $ s \geq \frac{3}{2}(1-\alpha)$. By Lemma \ref{lemproduct}, $u^{\epsilon}~\in~H^{s+\frac{\alpha}{2}}(\mathbb{R})$ ensures $(u^{\epsilon})^p~\in~H^{s+\frac{\alpha}{2}}(\mathbb{R})$, as the  condition $ s\geq \frac{3}{2}(1 - \alpha) $ guarantees that $H^{s+\frac{\alpha}{2}}(\mathbb{R})$ is an algebra.

\noindent
The estimation of the second term in RHS of the eq. \eqref{wp2} is given as
\begin{eqnarray} \label{local2}
  \int_{\mathbb{R}} |u^{\epsilon}_x J^su^{\epsilon} J^s (u^{\epsilon})^p| dx &\leq&  \|u^{\epsilon}_x\|_{L^{q_3}}\|J^s u^{\epsilon}\|_{L^{q_4}} \|J^s (u^{\epsilon})^p\|_{L^{q_5}} \nonumber\\
  &\leq& C \|u^{\epsilon}_x\|_{H^{s+\frac{\alpha}{2}-1}}  \|J^s u^{\epsilon} \|_{H^{\frac{\alpha}{2}}}  \|J^s (u^{\epsilon})^p\|_{H^{\frac{\alpha}{2}}} \nonumber\\
  &\leq & C \|u^{\epsilon}\|^{p+2}_{H^{s+\frac{\alpha}{2}}},
\end{eqnarray}
where $\frac{1}{q_3} + \frac{1}{q_4} + \frac{1}{q_5} \leq 1$.
Here the Sobolev imbeddings
\begin{eqnarray*}
  H^{s+\frac{\alpha}{2}-1} & \hookrightarrow & L^{q_3}, ~~\mbox{for}~~ {q_3} \leq \frac{ 2}{3-2s-\alpha},\\
  H^{\frac{\alpha}{2}} & \hookrightarrow & L^{q_4},~~\mbox{for}~~{q_4}\leq \frac{ 2}{1-\alpha},\\
  H^{\frac{\alpha}{2}} & \hookrightarrow & L^{q_5},~~\mbox{for}~~{q_5}\leq \frac{2}{1-\alpha},
\end{eqnarray*}
provide $ s\geq \frac{3}{2}(1 - \alpha) $.


\noindent
The estimation of  last term of the eq. \eqref{wp2},
\begin{equation*}
  \int_{\mathbb{R}}| R J^s u^{\epsilon} | dx  \leq \| R \|_{L^{2/(1+\alpha)}} \|J^s u^{\epsilon}\|_{L^{2/(1-\alpha)}} \leq C \|R\|_{L^{2/(1+\alpha)}}  \|u^{\epsilon}\|_{H^{s+\frac{\alpha}{2}}},
\end{equation*}
 follows directly from
H{\"o}lder's inequality and the Sobolev imbedding $\displaystyle H^{\frac{\alpha}{2}}\hookrightarrow L^{2/(1-\alpha)}$. Following \cite{linares,kenig} and using
\begin{eqnarray*}
H^{\frac{\alpha}{2}+\mu}(\mathbb{R})  & \hookrightarrow & L^{4/(1+\alpha)}, ~~~~  \frac{4}{1+\alpha}\leq \frac{2}{1-2(\frac{\alpha}{2}+\mu)},\\
 H^{s+\frac{\alpha}{2}-1-\mu}(\mathbb{R})  & \hookrightarrow & L^{4/(1+\alpha)}, ~~~~\frac{4}{1+\alpha}\leq \frac{2}{1-2(s+\frac{\alpha}{2}-1-\mu)}
 \end{eqnarray*}
for any $0<\mu <s $, one gets
\begin{equation*}
  \|R\|_{L^{2/(1+\alpha)}} \leq C \| J^{s-\mu}(u^{\epsilon})^p\|_{L^{4/(1+\alpha)}} \|J^{\mu}(u^{\epsilon})_x\|_{L^{4/(1+\alpha)}} \leq  C \|u^{\epsilon}\|_{H^{s+\frac{\alpha}{2}}}^p \|u^{\epsilon}\|_{H^{s+\frac{\alpha}{2}}},
\end{equation*}

%
\noindent
and finally
\begin{equation}\label{local3}
  \int_{\mathbb{R}} | R J^s u^{\epsilon} | dx  \leq  C \|u^{\epsilon}\|^{p+2}_{H^{s+\frac{\alpha}{2}}}.
\end{equation}
Choosing a suitable $\mu$, last restriction provides $ s\geq \frac{3}{2}(1-\alpha) $ as above.
Combining the eqs.  \eqref{local1}, \eqref{local2} and \eqref{local3}, we have
\begin{equation*}
  \frac{d}{dt} \|J^{s+\frac{\alpha}{2}} u^{\epsilon}(t)\|^{2}_{L^{2}} \leq C \|J^{s+\frac{\alpha}{2}} u^{\epsilon}(t)\|^{p+2}_{L^{2}}.
\end{equation*}
From \mbox{$\| J^s u \|_{L^{2}}=   \| u \|_{H^{s}}$},  the energy estimate is given by
\begin{equation}
  \frac{d}{dt} \| u^{\epsilon}(t)\|^{2}_{H^{r}} \leq C \| u^{\epsilon}(t)\|^{p+2}_{H^{r}}.
\end{equation}
Since $\|u^{\epsilon} (t) \|^2_{H^{r}} \leq y(t)$ where $y(t)$ is the solution of the following differential equation
\begin{eqnarray}
&& y^{\prime} (t)=C [ y (t) ]^{\frac{p+2}{2}} \nonumber \\
&& y(0)=\| u_0^{\epsilon} \|_{H^{r}}^2.
\end{eqnarray}
The energy bound  is given by
\begin{equation}
y(t)=\frac{y(0)}{\{2-Cp[y(0)]^{p/2} ~t \}^{2/p}}= \frac{\| u_0^{\epsilon} \|^2_{H^{r}}}{ \{ 2-Cp\| u_0^{\epsilon} \|_{H^{r}}^p ~t \}^{2/p} }.
\end{equation}

\begin{theorem}
Let $0<\alpha <1, r \geq  2 - \frac{\alpha}{2} $ and $u_0 \in H^r(\mathbb{R})$. Then there exists a time $T>0$, the solution $u^\epsilon$ of  the Cauchy problem for the eqs. \eqref{reg1}-\eqref{regin} converges uniformly to
$u$ of the Cauchy problem for the eqs. \eqref{gfBBM}-\eqref{incon} in $C([0,T],H^{r}(\mathbb{R}))$  as $\epsilon\rightarrow 0$.
\end{theorem}
\begin{proof}
We show that $(u^{\epsilon}(t))_{\epsilon \geq 0}$ is a Cauchy sequence in $[0,T]$. Let $\epsilon, \delta \geq 0$ and $u^{\epsilon}, v^{\delta}$ be the respective solutions of the eqs. \eqref{reg1}-\eqref{regin}. The difference
$w=u^{\epsilon}-v^{\delta}$ satisfies
\begin{equation}
w_t+w_x+ \frac{3}{4} D^{\alpha} w_x + \frac{3}{4} D^{\alpha} w_t +\frac{1}{2} \big\{ [(u^{\epsilon})^{p+1}]_x -  [(v^{\delta})^{p+1}]_x \big\} = (\epsilon-\delta)u^{\epsilon}_{xx} +\delta w_{xx}. \label{dif}
\end{equation}
Multiplying both sides of the eq. \eqref{dif} by $w$ and then integrating on the whole line, we have
\begin{eqnarray}
&& \hspace*{-10pt} \frac{d}{dt} \int_{\mathbb{R}}[ w^2+(D^{\alpha /2} w)^2] dx   \nonumber \\
&&  = - \int_{\mathbb{R}} \Big\{   w \big[ (u^\epsilon)^p +(u^\epsilon)^{p-1}v^{\delta}+....+(v^\delta)^p  \big]  \Big\}_x w dx  + \int_{\mathbb{R}} \big[ 2(\epsilon-\delta) u^{\epsilon}_{xx} w -2\delta (w_x)^2  \big] dx   \nonumber \\
&& \leq   \int_{\mathbb{R}}     \big[ (u^\epsilon)^p +(u^\epsilon)^{p-1}v^{\delta}+....+(v^\delta)^p  \big]  w w_x dx  + \int_{\mathbb{R}} \big[ 2(\epsilon-\delta) u^{\epsilon}_{xx} w \big] dx  \nonumber \\
&& \leq \frac{1}{2} \int_{\mathbb{R}}  \mid   \big[ (u^\epsilon)^p +(u^\epsilon)^{p-1}v^{\delta}+....+(v^\delta)^p  \big]_x  w^2 \mid  dx  +  2(\epsilon-\delta) \int_{\mathbb{R}} \mid   u^{\epsilon}_{xx} w  \mid dx.  \nonumber \\
\end{eqnarray}
By using the H{\"o}lder's inequality,  the first term of RHS is estimated as
\begin{eqnarray}
 &&  \hspace*{-30pt} \int_{\mathbb{R}}  \mid \big[ (u^\epsilon)^p +(u^\epsilon)^{p-1}v^{\delta}+....+(v^\delta)^p  \big]_x  w^2 \mid dx
   \nonumber  \\
 && \hspace*{30pt} \leq C \| ~ [(u^\epsilon)^p +(u^\epsilon)^{p-1}v^{\delta}+....+(v^\delta)^p ]_x ~ \|_{L^{q_1}}
 \|w\|_{L^{q_2}}^2 \nonumber\\
  &&  \hspace*{30pt} \leq C \| (u^\epsilon)^p +(u^\epsilon)^{p-1}v^{\delta}+....+(v^\delta)^p  \|_{H^{s+\frac{\alpha}{2}}}  \|w\|^2_{H^{\frac{\alpha}{2}}}. \nonumber \\
\end{eqnarray}
Here, we have used the  Sobolev imbeddings \eqref{embedding1}. By Theorem \ref{Thm1}, $u^{\epsilon}$ and  $v^{\delta}$ are bounded. Then, we deduce that
\begin{equation}
 \int_{\mathbb{R}}  \mid \big[ (u^\epsilon)^p +(u^\epsilon)^{p-1}v^{\delta}+....+(v^\delta)^p  \big]_x  w^2 \mid dx   \leq  C \|w\|^2_{H^{\frac{\alpha}{2}}}. \label{es1}
   \nonumber
\end{equation}
The second term of RHS is estimated as
\begin{eqnarray}
\int_{\mathbb{R}} \mid   u^{\epsilon}_{xx} w  \mid dx & \leq &  \| u_{xx}^{\epsilon} \|_{L^{q_6}}   \|w\|_{L^{q_7}}  \nonumber \\
&\leq & \| u_{xx}^{\epsilon} \|_{H^{s+\frac{\alpha}{2}-2}}  \|w\|_{H^{\frac{\alpha}{2}}}  \nonumber \\
&\leq & \| u^{\epsilon} \|_{H^{s+\frac{\alpha}{2}}}  \|w\|_{H^{\frac{\alpha}{2}}} .
\end{eqnarray}
Here, we have used the following Sobolev imbeddings
\begin{eqnarray}
  H^{s+\frac{\alpha}{2}-2} & \hookrightarrow & L^{q_6},~~\mbox{for}~~{q_6}\leq \frac{ 2}{5-2s-\alpha}, \nonumber\\
  H^{\frac{\alpha}{2}} & \hookrightarrow & L^{q_7},~~\mbox{for}~~{q_7}\leq \frac{2}{1-\alpha},   \nonumber
\end{eqnarray}
where $\frac{1}{q_6} + \frac{1}{q_7} \leq 1$ implies that $s\geq 2-\alpha$. Since $r=s+\frac{\alpha}{2}$, we  obtain $r \geq 2 - \frac{\alpha}{2}$. By Theorem \ref{Thm1}, it follows that
\begin{equation}
\int_{\mathbb{R}} \mid   u^{\epsilon}_{xx} w  \mid dx  \leq  C \|w\|_{H^{\frac{\alpha}{2}}}.  \label{es2}
\end{equation}
Combining the estimates \eqref{es1} and \eqref{es2},  we get
\begin{equation}
\frac{d}{dt} \|w\|_{H^{\frac{\alpha}{2}}}^2 \leq C  \|w\|^2_{H^{\frac{\alpha}{2}}} +C (\epsilon-\delta) \|w\|_{H^{\frac{\alpha}{2}}} \nonumber
\end{equation}
The Gronwall lemma implies that $(u^{\epsilon}(t))_{\epsilon \geq 0}$ is a Cauchy sequence in the complete space $H^{\frac{\alpha}{2}} (\mathbb{R})$ and it converges to a limit $u(t)$. Moreover,
$u^{\epsilon}(t)$ is continuous with respect to time and uniformly bounded by Theorem \ref{Thm1}, the sequence $u^{\epsilon}(t)$ is also weakly convergent in  $H^{s+\frac{\alpha}{2}} (\mathbb{R})$
to the limit $u(t)$.
\end{proof}

\setcounter{equation}{0}
\section{Conserved Quantities}

In this section, we derive the conserved quantities of the gfBBM eq.
 \eqref{gfBBM} for the smooth enough solutions which tend to 0 as $x\rightarrow \mp \infty$.
  As \mbox{$\displaystyle (I+ \frac{5}{4}D^{\alpha})u=0$} implies that $u=0$ the operator $\displaystyle (I+ \frac{5}{4}D^{\alpha})$ is invertible. Thus the equation is rewritten in the conservative like form
\begin{equation} \label{gfBBM_v}
  u_{t} + \partial_x(I+ \frac{5}{4}D^{\alpha})^{-1} [u+\frac{1}{2}u^{p+1}+   \frac{3}{4}D^{\alpha} u]=0.
\end{equation}
Therefore, the first conserved integral  is
\begin{equation}\label{cq1}
  I_0= \int_{\mathbb{{R}}} u(x,t)dx.
\end{equation}
Multiplying the eq. \eqref{gfBBM} by $u$ and  integrating  on the whole line, we have
\begin{eqnarray}\label{int2_gfBBM}
  \int_{\mathbb{{R}}} (\frac{u^2}{2})_t dx+ \int_{\mathbb{{R}}}(\frac{u^2}{2})_x dx &+&
   \frac{1}{2} \int_{\mathbb{{R}}}  \frac{p+1}{p+2} (u^{p+2})_x dx \nonumber \\
   &+& \frac{3}{4}\int_{\mathbb{{R}}} u D^{\alpha} u_{x}  dx+ \frac{5}{4}\int_{\mathbb{{R}}} u D^{\alpha} u_{t}  dx =0. \nonumber \\
  \end{eqnarray}
Thanks to the Plancherel theorem, one can write
\begin{eqnarray} \label{planc1}
  \int_{\mathbb{{R}}} uD^{\alpha} u_{t} dx  &=& \int_{\mathbb{{R}}} |\xi|^{\alpha} \hat{u}_{t}(\xi,t)\overline{\hat{u}(\xi,t)} d \xi  , \nonumber \\
   &=&   \frac{1}{2}\frac{d}{dt} \int_{\mathbb{{R}}}
   |D^{\frac{\alpha}{2}} u|^2 dx.
\end{eqnarray}
Rewriting the eq.  \eqref{int2_gfBBM} in the form
\begin{equation*}
  \frac{1}{2} \frac{d}{dt} \int_{\mathbb{{R}}} (u^2 + \frac{5}{4}|D^{\frac{\alpha}{2}} u|^2) dx + \frac{1}{2} \int_{\mathbb{{R}}}(u^2+ \frac{p+1}{p+2} u^{p+2} +\frac{3}{4} |D^{\frac{\alpha}{2}} u |^2)_x dx =0
  \end{equation*}
allows us to write a second conserved integral
\begin{equation} \label{cq2}
  I_1= \int_{\mathbb{{R}}} (u^2 + \frac{5}{4}|D^{\frac{\alpha}{2}}u |^2) dx
\end{equation}
for any $\alpha$.  Therefore, the Cauchy problem for the eq. \eqref{gfBBM} admits a unique global weak solution in $L^\infty(\mathbb{R},H^{\frac{\alpha}{2}}(\mathbb{R}))$.

The equivalent form for  the eq. \eqref{gfBBM_v}  gives us the Hamiltonian formulation
\begin{equation*}
  \partial_t u + J_{\alpha} \nabla_u H (u)=0.
\end{equation*}
Here, skew-adjoint operator $J_{\alpha}$ is
\begin{equation*}\label{J}
   J_{\alpha}= \partial_x(I+ \frac{5}{4}D^{\alpha})^{-1},
\end{equation*}
and the Hamiltonian is
\begin{equation} \label{hamiltonian}
  H(u)=\frac{1}{2} \int_{\mathbb{{R}}} \big( u^2 + \frac{u^{p+2}}{p+2}+ \frac{3}{4} |D^{\frac{\alpha}{2}}u|^2 \big) dx.
\end{equation}
The Sobolev Imbedding  $H^{\alpha/2} \hookrightarrow  L^{p+2}$ yields that   $ H(u)$ is well-defined for  \mbox{$\displaystyle\alpha \geq \frac{p}{p+2}$}.
\setcounter{equation}{0}
\section{Solitary Wave Solutions}
To find the  localized solitary wave solutions of the eq. \eqref{gfBBM}, we use the ansatz $u(x,t)=Q_c(\xi), ~~\xi=x-ct$
with $\displaystyle{\lim_{|\xi| \rightarrow \infty} Q_c(\xi)=0}$ which leads to  the ordinary differential equation
\begin{equation*}
  -cQ_c' + Q_c' + \frac{1}{2}(Q_c^{p+1})' + \frac{3}{4} D^{\alpha} Q_c' -c\frac{5}{4}D^{\alpha} Q_c' =0.
\end{equation*}
Here $^\prime$ denotes the derivative with respect to $\xi$.
Integrating the above equation, we have
\begin{equation}\label{sw1}
  \big(\frac{5}{4}c-\frac{3}{4}  \big) D^{\alpha} Q_c + (c-1) Q_c - \frac{1}{2} (Q_c)^{p+1}=0.
\end{equation}
The following theorem shows the non-existence of the nontrivial solutions of the eq.  \eqref{sw1} for some values of $\alpha, p$ and $c$.
\begin{theorem} \label{Thm2}
 Assume that one of the following cases
  \begin{description}
    \item[i.] $c \in (\frac{3}{5},1)$ {and} $\alpha \geq \frac{p}{p+2}$,
    \item[ii.]  $c \notin (\frac{3}{5},1)$ {and} $\alpha \leq \frac{p}{p+2}$,
    \item[iii.] $c=\frac{3}{5}$ {or} $c=1$,
  \end{description}
is satisfied. Then, the eq. \eqref{sw1} does not admit any nontrivial solution $Q_c \in H^{\frac{\alpha}{2}}(\mathbb{R})\bigcap L^{p+2}(\mathbb{R})$.
\end{theorem}
\textbf{Proof:} Let $Q_c$ be any nontrivial solution of the eq. \eqref{sw1} in the class $H^{\frac{\alpha}{2}}(\mathbb{R})\bigcap L^{p+2}(\mathbb{R})$. Multiplying the eq. \eqref{sw1} by $Q_c$ and integrating on $\mathbb{R}$, we get
\begin{equation}
  \big(\frac{5}{4} c-\frac{3}{4}  \big) \int_{\mathbb{{R}}} Q_c D^{\alpha} Q_cdx + (c-1) \int_{\mathbb{{R}}} Q_c^2 dx = \frac{1}{2} \int_{\mathbb{{R}}} Q_c^{p+2} dx. \label{ph1}
\end{equation}
By  using the Plancherel's formula
 \begin{eqnarray*}
  \int_{\mathbb{{R}}} Q_cD^{\alpha} Q_c dx &=& \int_{\mathbb{{R}}} |\xi|^{\alpha} \hat{Q_c}(\xi) \overline{\hat{Q_c}(\xi)} d \xi, \\
    &=& \int_{\mathbb{{R}}} |D^{\frac{\alpha}{2}} Q_c|^2 dx,
 \end{eqnarray*}
 \eqref{ph1} becomes
\begin{equation}\label{energy1}
  \big(\frac{5}{4} c-\frac{3}{4} \big) \int_{\mathbb{{R}}} |D^{\frac{\alpha}{2}} Q_c|^2 dx + (c-1)\int_{\mathbb{{R}}} Q_c^2 dx = \frac{1}{2} \int_{\mathbb{{R}}} Q_c^{p+2} dx.
\end{equation}
On the other hand, multiplying  the eq. \eqref{sw1} by $x Q_c'$ and integrating over $\mathbb{R}$,  we have
\begin{equation} \label{sw2}
  \big(\frac{5}{4} c-\frac{3}{4} \big) \int_{\mathbb{{R}}} x Q_c' D^{\alpha} Q_c dx + (c-1)\int_{\mathbb{{R}}} x Q_c' Q_c dx - \frac{1}{2} \int_{\mathbb{{R}}} x Q_c' Q_c^{p+1} dx =0.
\end{equation}

\noindent
Later, the equality
\begin{equation} \label{id3}
  \int_{\mathbb{{R}}} x Q_c' D^{\alpha} Q_c dx = \frac{\alpha-1}{2} \int_{\mathbb{{R}}} |D^{\frac{\alpha}{2}}Q_c|^2 dx,
\end{equation}
in \cite{linares} and several integration by parts
turns the eq. \eqref{sw2} into the Pohozaev type identity
\begin{equation} \label{pi}
  \big(\frac{5}{4}c-\frac{3}{4}\big)\frac{1-\alpha}{2} \int_{\mathbb{{R}}}  |D^{\frac{\alpha}{2}} Q_c|^2 dx + \frac{c-1}{2}\int_{\mathbb{{R}}}  Q_c^2 dx - \frac{1}{2(p+2)} \int_{\mathbb{{R}}} Q_c^{p+2} dx =0.
\end{equation}
Finally, by gathering  the eqs. \eqref{energy1} and  \eqref{pi}, we obtain
\begin{equation}\label{pohozaevlast}
\int_{\mathbb{{R}}}  |D^{\frac{\alpha}{2}} Q_c|^2 dx=\frac{4p(c-1)}{(5c-3 ) [\alpha(p+2)-p  ] }\int_{\mathbb{{R}}}Q_c^2dx
\end{equation}
and the  proof of items (\textbf{i}) and (\textbf{ii}) of the theorem follows directly from the positivity of the left hand integral. To prove item (\textbf{iii}) we first set $c=1$ in \eqref{pohozaevlast} which gives $D^{\frac{\alpha}{2}} Q_c=0$ and then the assumption $\displaystyle{\lim_{|x| \rightarrow \infty} Q_c(x)=0}$ shows that the solution is trivial. Lastly, if we rewrite \eqref{pohozaevlast} as
  \begin{equation*}
\frac{(5c-3 ) [\alpha(p+2)-p  ] }{4p(c-1)}\int_{\mathbb{{R}}}  |D^{\frac{\alpha}{2}} Q_c|^2 dx=\int_{\mathbb{{R}}}Q_c^2dx
\end{equation*}
and set $c=3/5$  we directly have that $Q_c=0$, which completes the proof.

\noindent Combining the results of Theorem 4.1 and  the condition $\displaystyle \alpha \geq \frac{p}{p+2}$ which ensures that the Hamiltonian \eqref{hamiltonian} is well-posed for $0< \alpha < 1$ we conclude that in order to have a non-trivial solution we must have $\displaystyle \alpha > \frac{p}{p+2}$ and $c<3/5$ or $c>1$.

\noindent On the other hand if we assume that the solution $Q_c$ is positive and $c<3/5$ then the eq. \eqref{ph1}
 gives a contradiction while the RHS of the equation is positive and the LHS is negative. Therefore we are able to say that the eq. \eqref{sw1} has no positive solitary wave solutions unless $c>1$.


\par In order to show the existence and uniqueness of the solitary wave solutions, we recall the results of Frank and Lenzmann \cite{franklenzmann}.
\begin{definition}\label{groundstates} (Definition 2.1 of \cite{franklenzmann}, Definition 1.1 of \cite{pava}) Let $Q\in H^{\alpha/2}(\mathbb{R})$ be an even and positive solution of the equation
\begin{equation}\label{ground}
D^{\alpha} Q+Q-Q^{p+1}=0.
\end{equation}
If
\begin{equation}
J^{\alpha,p}(Q)=\inf \{J^{\alpha,p}(u)|u\in H^{\alpha/2}(\mathbb{R}) \backslash\{0\}\}
\end{equation}
then $Q\in H^{\alpha/2}(\mathbb{R})$ is a ground state solution of the  eq. \eqref{ground} where $J^{\alpha,p}$ is the Weinstein functional defined by
\begin{equation*}
J^{\alpha,p}(u)=\left(\int_{\mathbb{R}} |u|^{p+2} dx  \right)^{-1}
\left(\int_{\mathbb{R}} |D^{\alpha/2}u|^{2}  dx  \right)^{p/2\alpha}
\left(\int_{\mathbb{R}}  |u|^{2} dx  \right)^{p(\alpha-1)/2\alpha+1}.
\end{equation*}
\end{definition}

\noindent
 The scaling
\begin{equation*}
Q_c(\xi)=(2(c-1))^{1/p}Q\left(\left( \frac{4(c-1)}{5c-3}  \right)^{1/\alpha}   \xi \right)
\end{equation*}

\noindent
converts  the eq. \eqref{sw1} into the eq. \eqref{ground}.  In Proposition $1.1$ and Theorem $2.4$ of \cite{franklenzmann}, Frank and Lenzmann prove the existence and uniqueness of the positive ground state solutions of
the  eq. \eqref{ground}  when $0<\alpha<2$ and  $0<p<p_{max}$ holds, where the critical exponent is defined as
\begin{equation}\label{pcond}
p_{\max}(\alpha)=\left\{ \begin{array}{cc}
\frac{2\alpha}{1-\alpha}, &\mbox{for}~~ 0< \alpha < 1  \\
\infty, &\mbox{for}~~ 1\leq \alpha < 2.   \end{array} \right.
\end{equation}

\noindent
Therefore, the eq. \eqref{sw1} has a unique positive ground state solution \mbox{$Q_c\in H^{\alpha/2}(\mathbb{R})$} for $c>1$ and $\displaystyle \alpha > \frac{p}{p+2}$. That is why in the following section, we choose the parameters
$c, \alpha$  and $p$  satisfying these  conditions to obtain positive solitary wave solutions, numerically.

\setcounter{equation}{0}
\section{Numerical results for gfBBM Equation}

In this section, we discuss the numerical solutions of the gfBBM equation. 
Since we do not know the analytical solitary wave  solutions of the eq. \eqref{gfBBM} for any $\alpha \in (0,1)$,  we first use the Petviashvili's iteration  method  \cite{pelinovsky, petviashvili, yang,uyen} to construct the solitary wave solution numerically. The solitary wave solution of the gfBBM equation satisfies the eq.  \eqref{sw1}. Applying   the Fourier transform to the eq. \eqref{sw1} yields
\begin{equation*}
  \widehat{Q_c}(k)[(\frac{5c}{4}-\frac{3}{4})|k|^{\alpha}+ c-1]=\frac{1}{2} \widehat{Q_c^{p+1}}(k).
\end{equation*}
We propose standard iterative algorithm in the form
\begin{equation}\label{pm1}
   \widehat{Q}_{n+1}(k)=\frac{1}{2[(\displaystyle\frac{5c}{4}-\frac{3}{4})|k|^{\alpha} + c-1]} \widehat{Q_{n}^{p+1}}(k),
\end{equation}
where $Q$ is used instead of $Q_c$ for simplicity. The condition $c>1$ guarantees  the non-resonance condition $(\displaystyle\frac{5c}{4}-\frac{3}{4})|k|^{\alpha} + c-1\neq0$ for any $k\in \mathbb{R}$.
The main idea in the Petviashvili method is to add a stabilizing factor into the fixed-point iteration in order to  prevent  the iterated solution to converge to zero solution or diverge.
Then, the Petviashvili method  for the gfBBM eq. is given by
\begin{equation}
   \widehat{Q}_{n+1}(k)=\frac{M_n^{\nu}}{2[(\displaystyle\frac{5c}{4}-\frac{3}{4})|k|^{\alpha} +c-1]}\widehat{Q_{n}^{p+1}}(k) \label{scheme}
\end{equation}
with stabilizing factor
\begin{equation*}
  M_n^{\nu}=\frac{\int_{\mathbb{R}} [(\displaystyle\frac{5c}{4}-\frac{3}{4})|k|^{\alpha} +c-1] [\widehat{Q}_{n}(k)]^2 dk }{\int_{\mathbb{R}}\displaystyle\frac{1}{2} \widehat{Q}^{p+1}_{n}(k) \widehat{Q}_{n}(k)dk },
\end{equation*}
for some parameter $\nu$. Pelinovsky et. al.  in  \cite{pelinovsky} showed that the fastest convergence occurs when $\nu=({p+1})/{p}$. Therefore to reduce the CPU time, we use  $\nu=({p+1})/{p}$ for the rest of the paper.
The Fourier pseudo-spectral method  is  employed  to implement the scheme \eqref{scheme}. The MATLAB functions \enquote{fft} and \enquote{ifft} compute the discrete Fourier transform and its inverse for any function $f(x)$ by using efficient Fast Fourier Transform.  We note that the Petviashvili iteration method can  also be used  for approximating  the periodic waves \cite{uyen}.

\par
Next, we investigate time evolution of the numerically generated solitary waves by using a numerical scheme combining a Fourier pseudo-spectral method for space and a fourth order Runge-Kutta method for the time integration.
Since the fractional derivative in the gfBBM equation is defined by a Fourier multiplier, the Fourier spectral method will be the most appropriate method for investigating the evolution of the solution in time.  We assume that $u(x,t)$ has periodic boundary condition $u(-L,t)=u(L,t)$ on the truncated domain $(x,t) \in [-L, L] \times [0,T]$.

First, the spatial period is normalized from finite interval $x\in[-L,L]$  to $X \in [0,2\pi]$ using the transformation \mbox{$X=\pi(x+L)/L$} in order to use the MATLAB functions \enquote{fft} and \enquote{ifft}.  In this case, the eq. \eqref{gfBBM} becomes
\begin{equation} \label{gfBBMX}
   \left[{\cal I}+ \frac{5}{4} (\frac{\pi}{L})^{\alpha} D^{\alpha}\right] u_t = -\frac{\pi}{L} u_X -  \frac{\pi}{2L} (u^{p+1})_X- \frac{3}{4} (\frac{\pi}{L})^{\alpha+1}  D^{\alpha} u_{X}.
\end{equation}
For the discretization of \eqref{gfBBMX} the interval $[0,2\pi]$ is divided into $N$ equal subintervals with grid spacing $\Delta X=2\pi/N$, where the integer $N$ is even. The spatial grid points are given by $X_{j}=2\pi j/N$,  $j=0,1,2,...,N$. The time interval $[0,T]$ is divided into $M$ equal subintervals with time step $\Delta t$. The temporal grip points are shown by $t_m=\frac{mT}{M}$,~ $m=0,...,M$.
The discrete Fourier transform of the sequence $\{U_{j} \}$, i.e.
\begin{equation}\label{dft}
  \widetilde{U}_{k}={\cal F}_{k}[U_{j}]=
          \frac{1}{N}\sum_{j=0}^{N-1}U_{j}e^{-ik X_{j}},
           ~~~~-N/2 \le k \le N/2-1
\end{equation}
gives the corresponding Fourier coefficients. Likewise, $\{U_{j} \}$ can be recovered from the Fourier coefficients by the inversion formula for the discrete Fourier transform (\ref{dft}), as follows:
\begin{equation}\label{invdft}
  U_{j}={\cal F}^{-1}_{j}[\widetilde{U}_{k}]=
          \sum_{\xi =-\frac{N}{2}}^{\frac{N}{2}-1}\widetilde{U}_{k}e^{i k X_{j}},
          ~~~~j=0,1,2,...,N-1~.
\end{equation}
Here $\cal F$ denotes the discrete Fourier transform and
${\cal F}^{-1}$ its inverse. Applying the discrete Fourier transform to the eq. \eqref{gfBBMX},  we get the ordinary differential equation given by
\begin{equation}
(\widetilde{U}_k)_t= \frac{\displaystyle -\frac{\pi}{L}ik \widetilde{U}_k-\frac{\pi}{2L}ik \widetilde{(U^{p+1})}_k-\frac{3}{4}(\frac{\pi}{L})^{\alpha+1}ik |k|^\alpha \widetilde{U}_k }{\displaystyle 1+\frac{5}{4} (\frac{\pi}{L})^{\alpha}|k|^\alpha}.
\end{equation}
We then use the fourth order Runge-Kutta method to solve the resulting ODE
in time. Finally, we find the approximate solution by using the inverse discrete Fourier transform.

\subsection{Numerical Generation of Solitary Waves}
\par A solitary wave solution $u(x,t)=Q_c(x-ct)$  of the equation
\begin{equation}
u_t+u u_x-D^\alpha u_x=0
\end{equation}
satisfies the ODE
\begin{equation}
D^\alpha Q_c+cQ_c-\frac{1}{2}Q_c^2=0. \label{cozum}
\end{equation}
The eq. \eqref{cozum} has the solution
\begin{equation}
Q_c(x,t)=\frac{4c}{1+c^2(x-ct)^2}
\end{equation}
for $\alpha=1$ in \cite{kleinfBBM}. By the convenient scaling,  the exact solitary wave solution of gfBBM equation can be written as
\begin{equation}\label{tws}
  Q_{exact}(x,t) = \frac{4(c-1)}{1+\big[\frac{4(c-1)}{5c-3}\big]^2 (x-ct)^2}
\end{equation}
for  $\alpha=1$ and $p=1$.

\noindent
In the first numerical experiment, we test our scheme by comparing the numerical result with the exact solution.
The space interval and number of grid points are chosen as $x \in [-2048,2048] $ and $N=2^{18}$, respectively.
The overall iterative process is  controlled by the  error,
\begin{equation}
  Error(n)=\|Q_n-Q_{n-1}\|,~~~~n=0,1,.... \nonumber
\end{equation}
 between two consecutive iterations defined with the  number of iterations,  the stabilization factor error
\begin{equation}
|1-M_n|, ~~~~n=0,1,.... \nonumber
\end{equation}
and the residual error
\begin{equation}
{RES(n)}= \|{\cal S} Q_n\|_\infty, ~~~~n=0,1,.... \nonumber
\end{equation}
where
\begin{equation}
{\cal S}Q= \big(\frac{3}{4}-  \frac{5c}{4} \big) D^{\alpha} Q - (1-c) Q - \frac{1}{2} Q^{p+1}.
\end{equation}

\noindent
In Figure $1$, we present the difference between the obtained numerical  and exact solitary wave solution and
the variation of    three different  errors
with the  number of iterations in semi-log scale. As it is seen from the Figure $1$, our proposed numerical scheme captures the solution remarkably well.
\noindent

\begin{figure}[ht]
 \begin{minipage}[t]{0.45\linewidth}
   \includegraphics[width=3.1in]{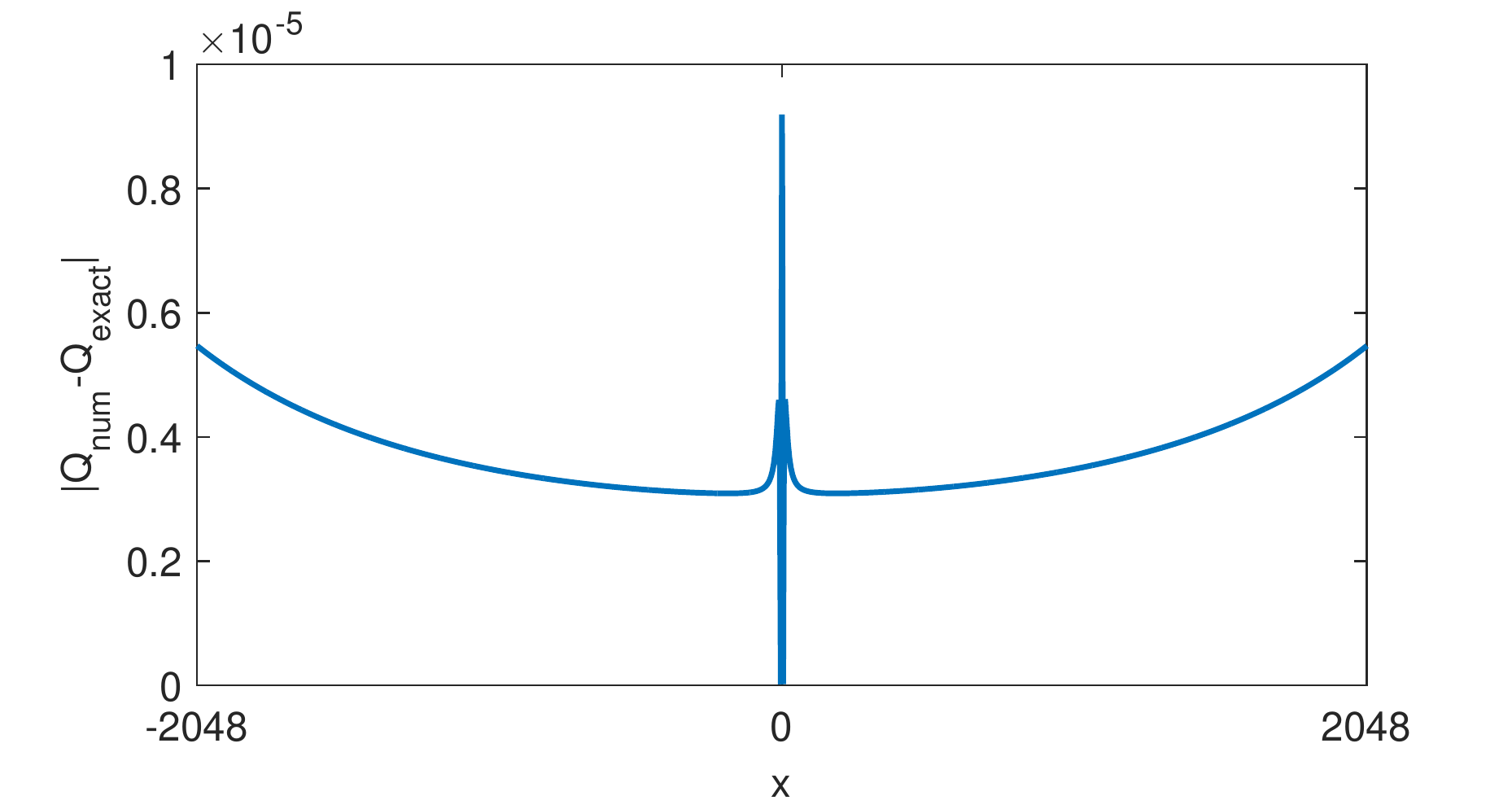}
 \end{minipage}
\hspace{30pt}
\begin{minipage}[t]{0.45\linewidth}
   \includegraphics[width=3in]{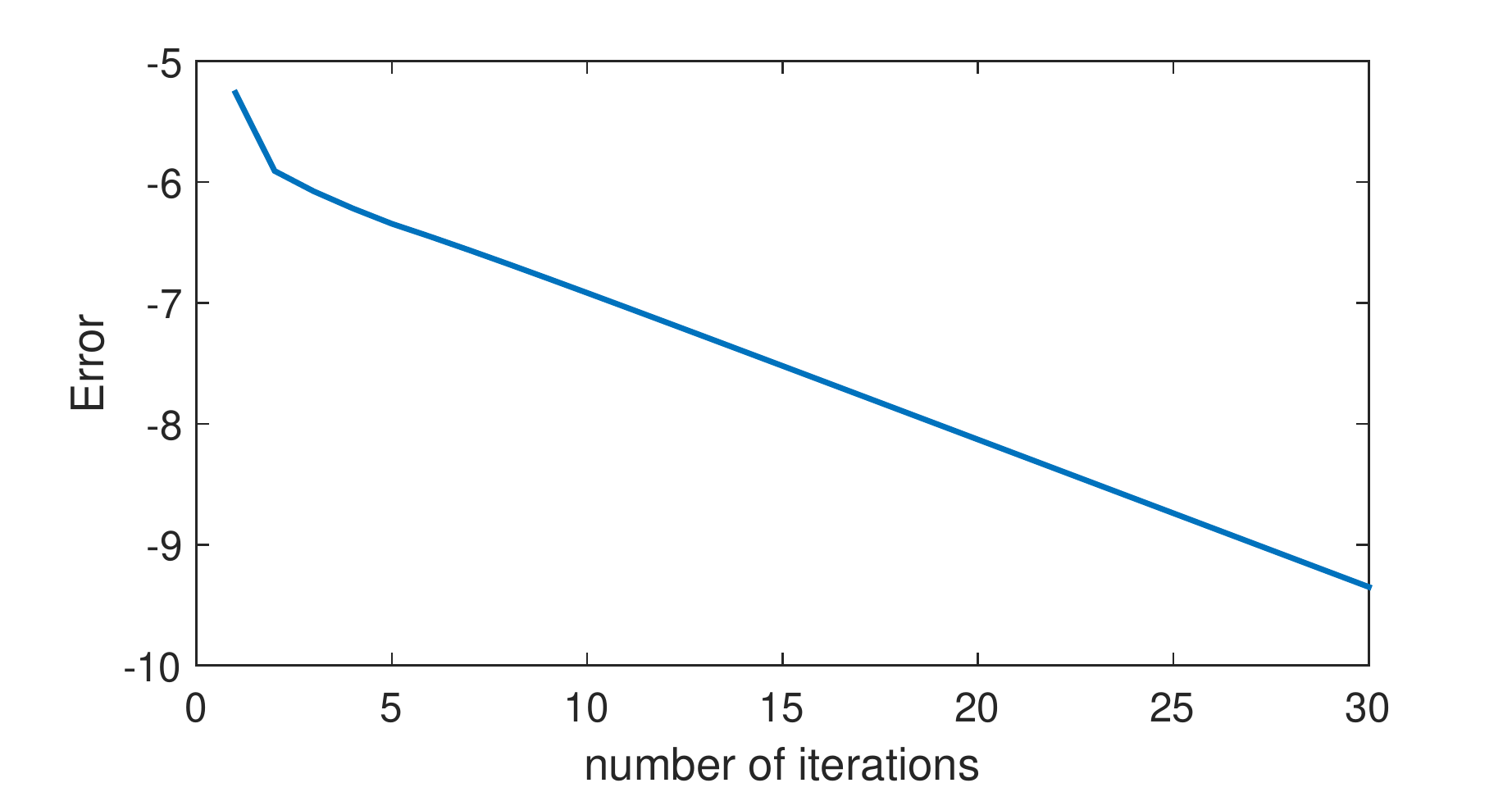}
 \end{minipage}
 \hspace{30pt}
\begin{minipage}[t]{0.45\linewidth}
   \includegraphics[width=3in]{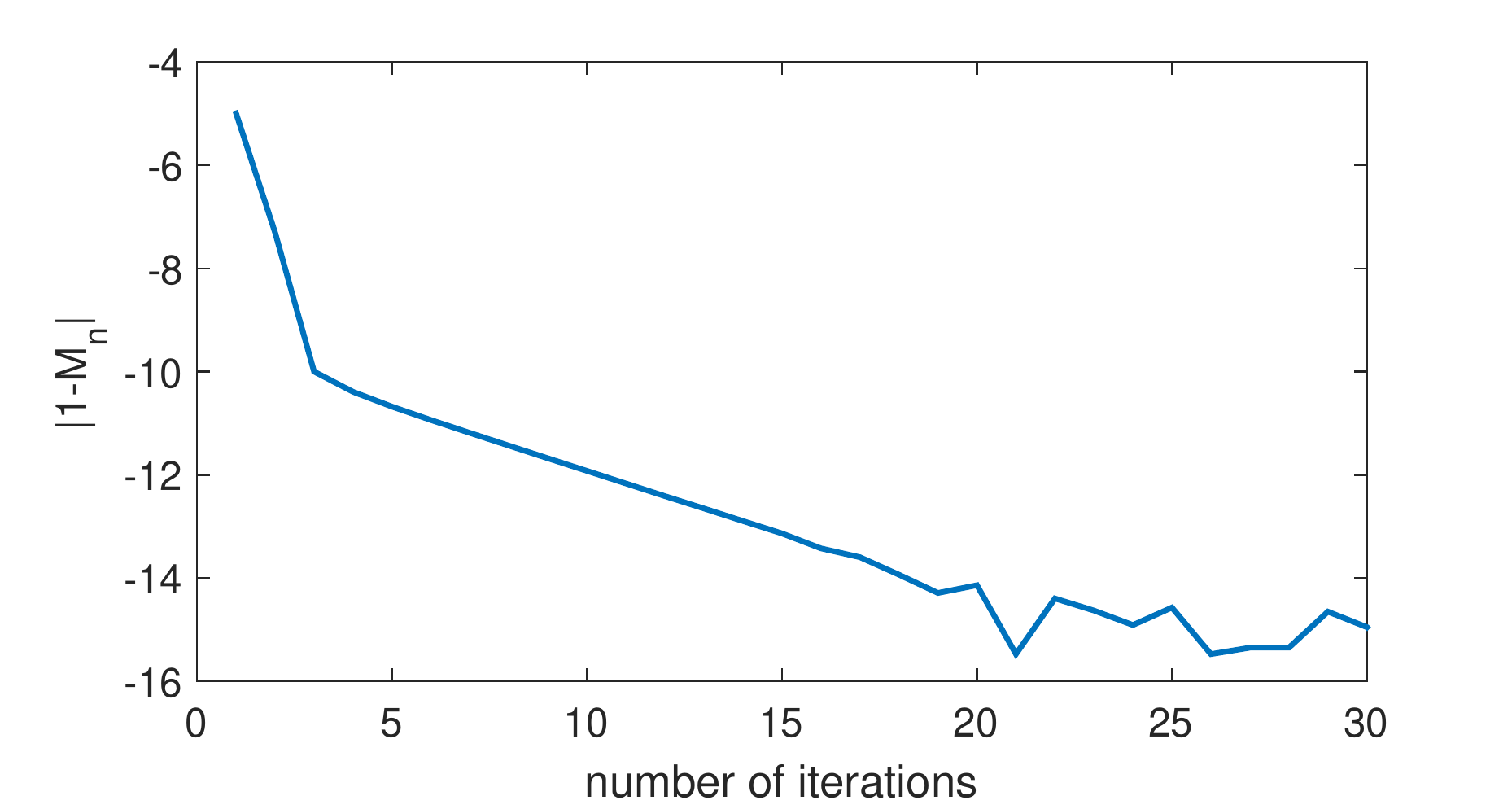}
 \end{minipage}
 \hspace{30pt}
\begin{minipage}[t]{0.45\linewidth}
   \includegraphics[width=3in]{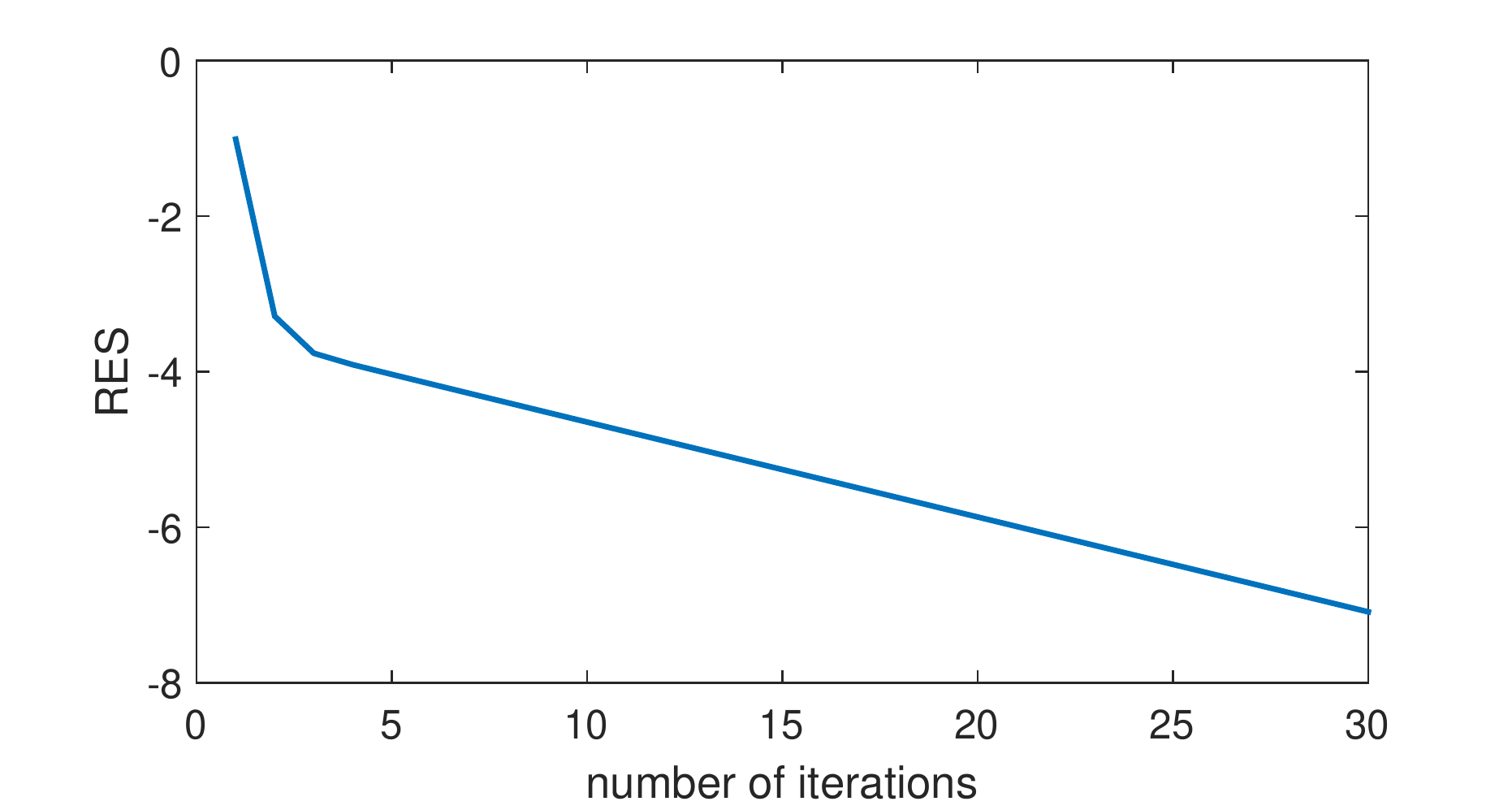}
 \end{minipage}
  \caption{Difference between the exact and the numerical solutions for  $\alpha=1$, $p=1$ and  the variation of the $Error(n)$, $|1-M_n|$ and $RES$ with the number of iterations in semi-log scale.}
 \label{pdif}
\end{figure}

\par In order to  understand the effects of the fractional dispersion, we illustrate the solitary wave profiles generated by Petviashvili's  iteration  method for various values of $\alpha$ and for $c=1.1,~p=1$  in Figure $2$. We observe that the solution becomes more and more peaked with decreasing values of $\alpha$.

\begin{figure}[!htbp]
 \centering
 \includegraphics[width=3.3in]{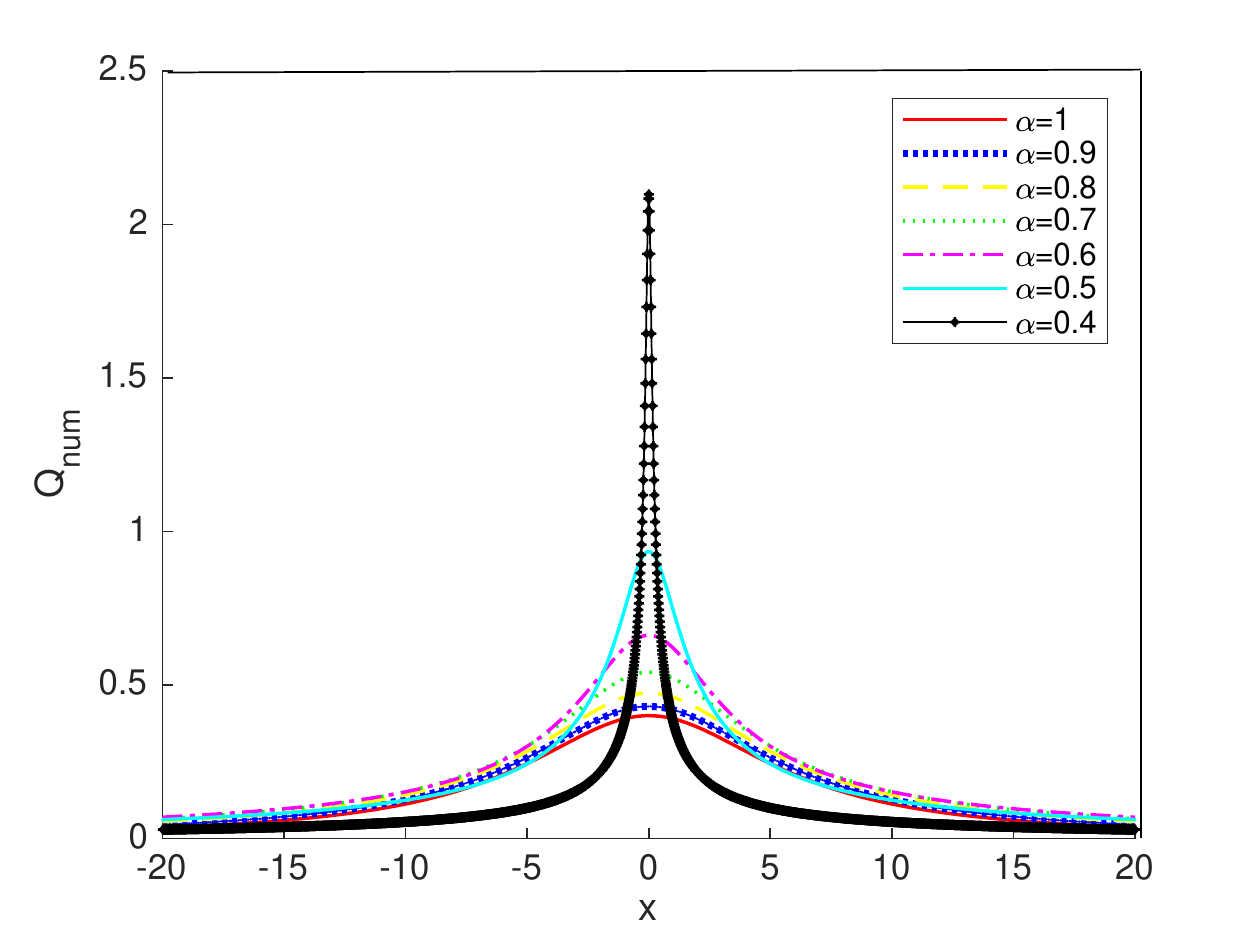}
  \caption{Solitary wave profiles  for various  values of $\alpha$ ($p=1$, $c=1.1$).}
\label{difalpha}
\end{figure}

\noindent
Since we do not know the exact solitary wave solution for different values of $\alpha$, we cannot compare the numerical solution with the exact solution.
Therefore, the iteration, stabilization factor and the residual errors are depicted   in Figure $3$, respectively, for $\alpha=0.6$ and $\alpha=0.8$.
These results show that the solitary wave solution generated by Petviashvili's method converges rapidly to the accurate solution.

\begin{figure}
\vspace{-2cm}
 \begin{minipage}[t]{0.4\linewidth}
   \includegraphics[width=2.7in]{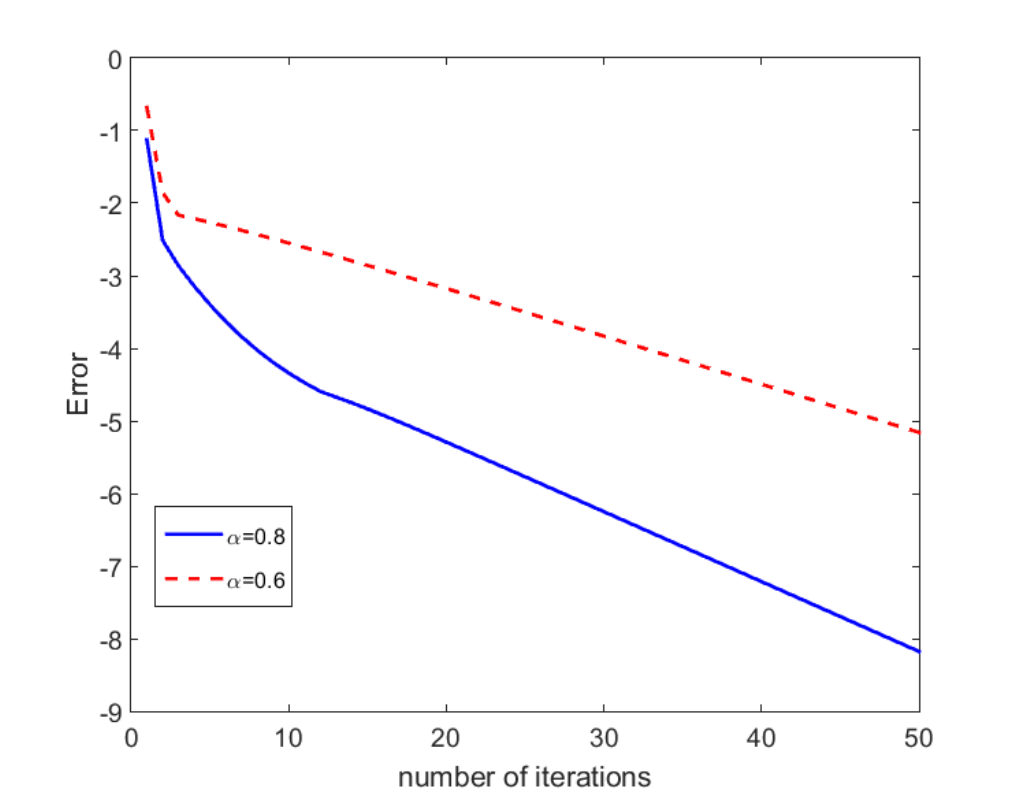}
 \end{minipage}
 \hspace{30pt}
 \begin{minipage}[t]{0.4\linewidth}
   \includegraphics[width=2.7in]{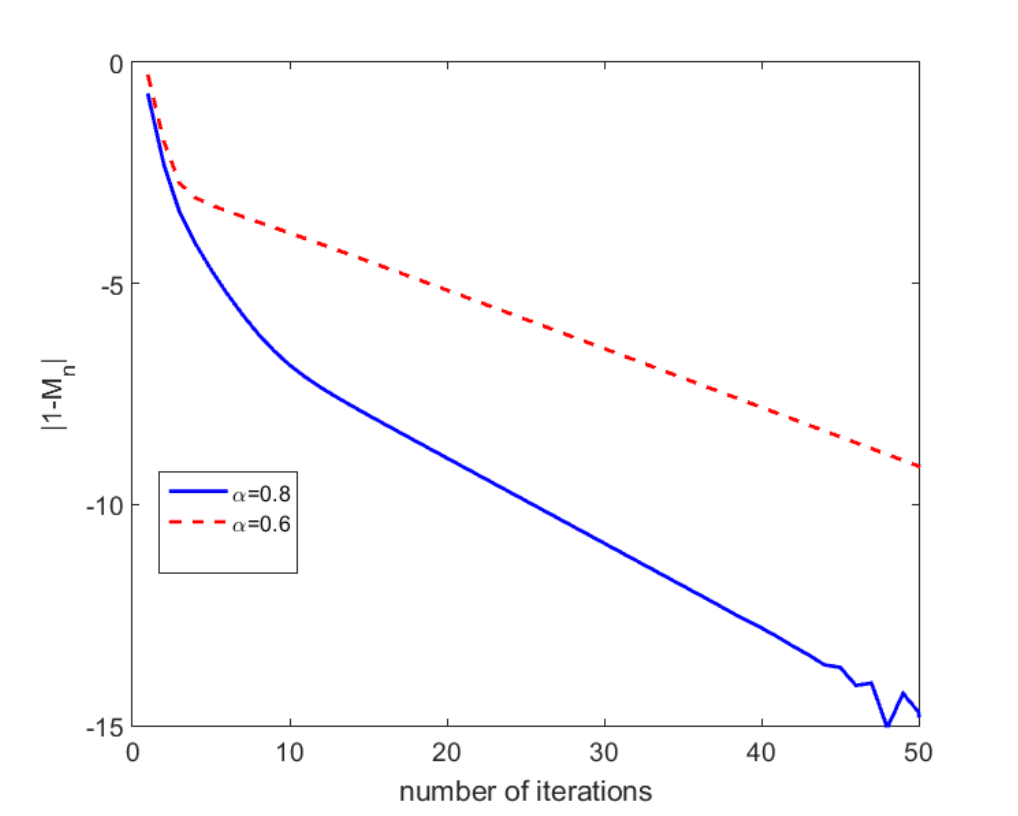}
 \end{minipage}
  \hspace{30pt}
  \begin{center}
 \begin{minipage}[t]{0.4\linewidth}
   \includegraphics[width=2.7in]{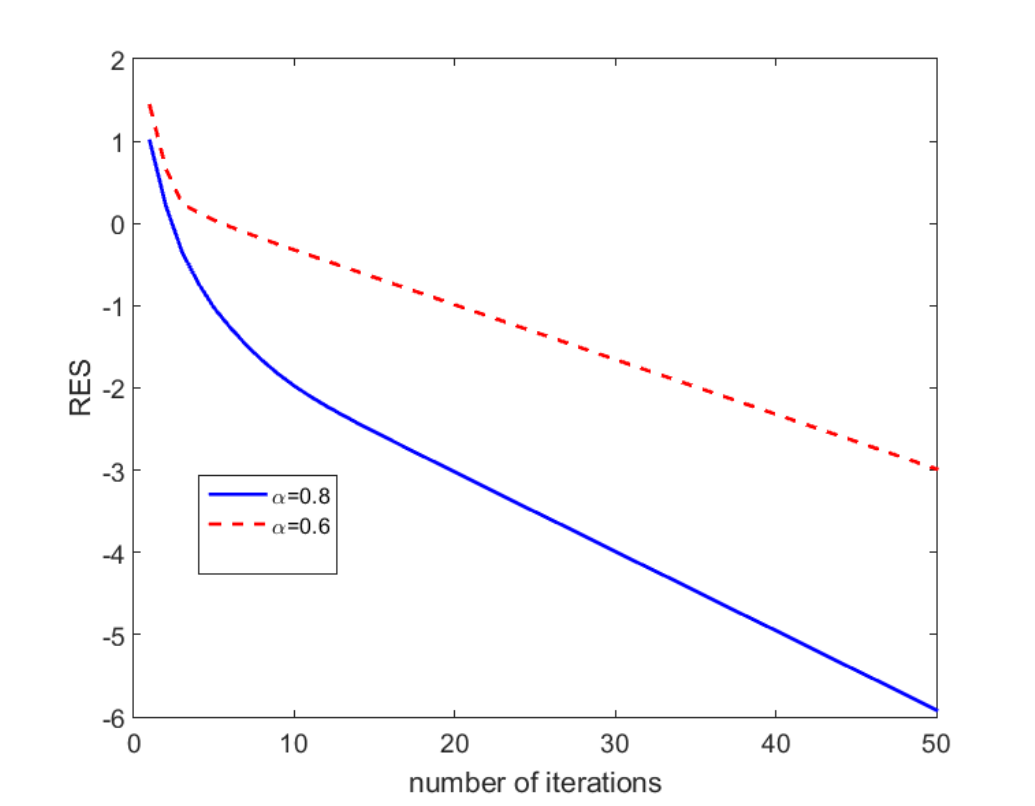}
 \end{minipage}
 \end{center}
  \caption{ The variation of the iteration, stabilization factor and the residual errors with the number of iterations in the semi-log scale ($c=1.1$, $p=1$).}
\label{errorsvsni}
\end{figure}

\begin{figure}
 \centering
 \includegraphics[width=3.3in]{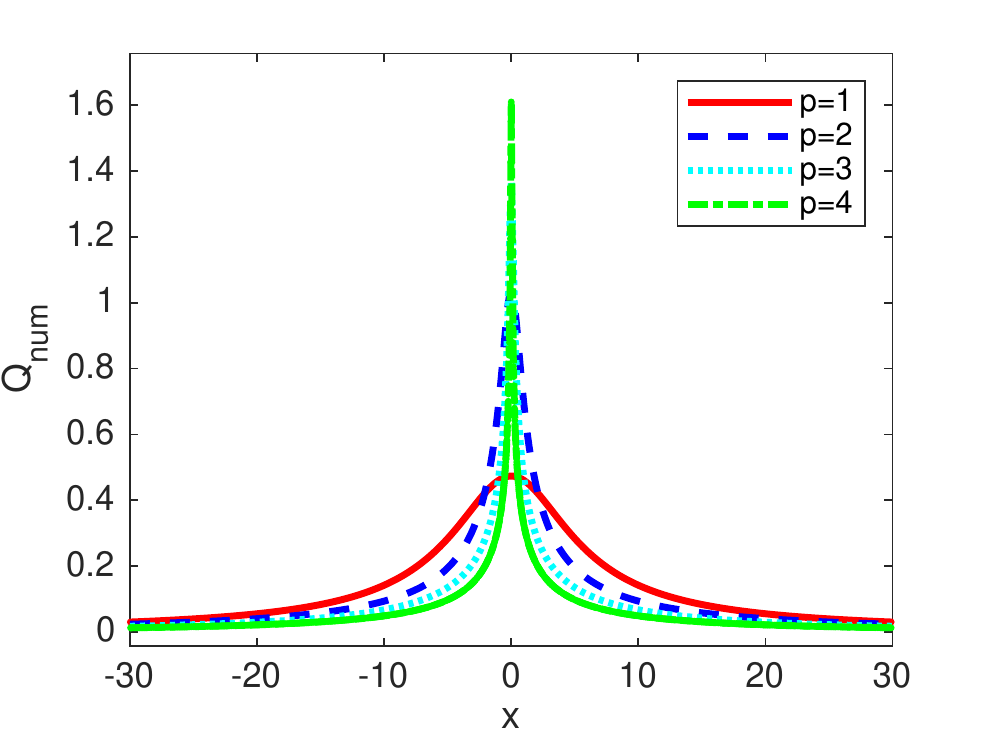}
  \caption{Solitary wave profiles  for various nonlinearities ($\alpha=0.8$, $c=1.1$).}
 \label{pdif}
\end{figure}

\noindent
In order to  understand the effects of the nonlinearity,    we present the solitary   wave profiles generated by Petviashvili's  iteration  method for various nonlinearities with  fixed $\alpha=0.8$ in Figure $4$. This numerical result agrees well with the fact that the wave steepens with increasing nonlinear effects.

In the next experiment we investigate the speed-amplitude relation. We illustrate the  variation of the amplitude with the speed parameter for various values of $p$ and the fixed $\alpha=0.8$ and various values of $\alpha$ and fixed nonlinearity $p=1$ in Figure $5$. We observe that the amplitude is increasing with the increasing values of speed, as expected for the solitary waves.
As it is seen from the left panel of Figure $5$, there is a critical speed $c_s$ near to $1.5$.  For a fixed value of $\alpha$ and speed  $(c < c_s)$, the amplitude increases  with  increasing nonlinearity. But, for a fixed value of  $\alpha$ and  speed $(c > c_s)$, the amplitude decreases with increasing nonlinearity. The right panel of Figure $5$ illustrates that the amplitude increases with decreasing values of $\alpha$ for a fixed value of nonlinearity $p=1$ and speed.

\begin{figure}[h!bt]
 \begin{minipage}[t]{0.45\linewidth}
   \includegraphics[width=3.2in]{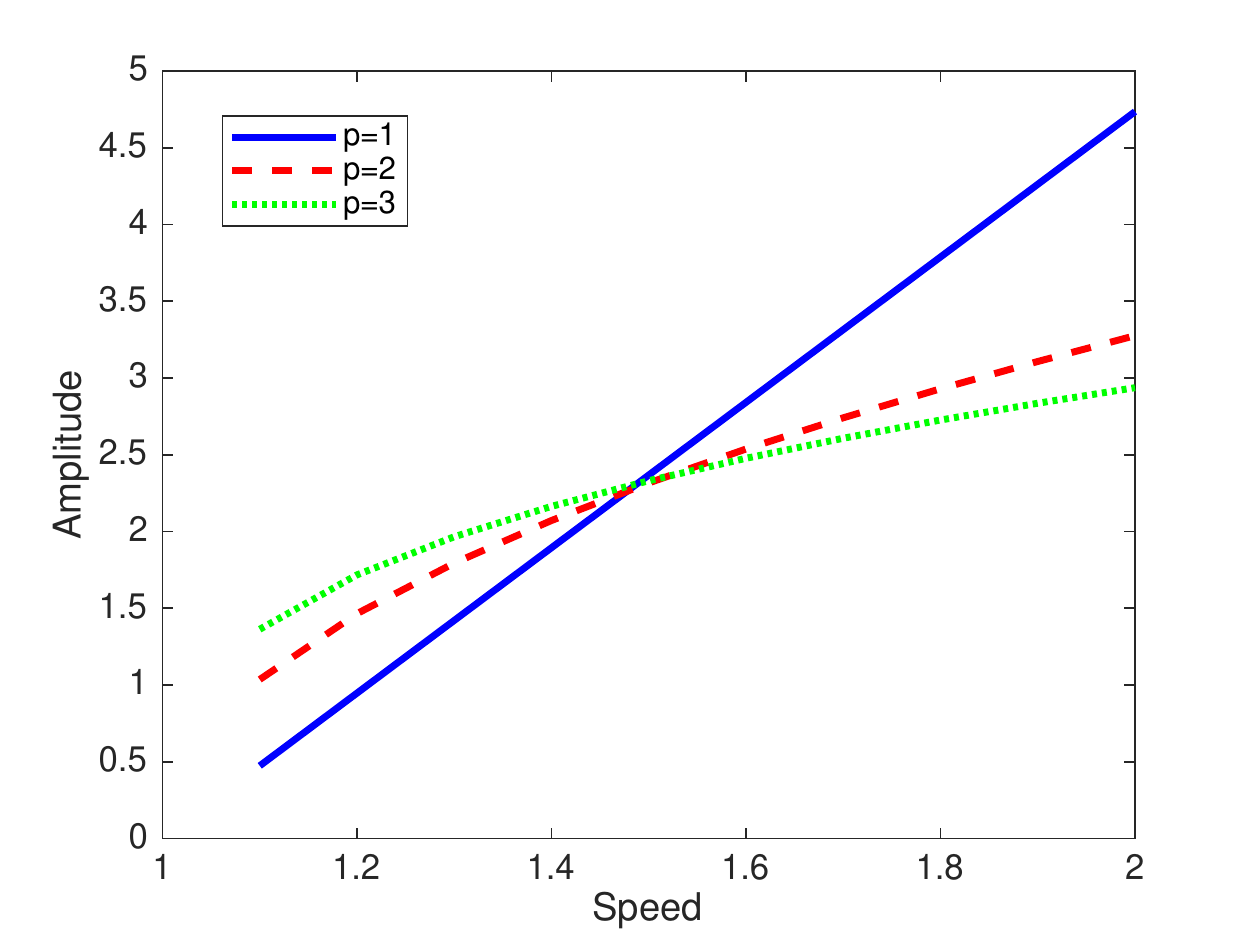}
 \end{minipage}
\hspace{30pt}
\begin{minipage}[t]{0.45\linewidth}
   \includegraphics[width=3.2in]{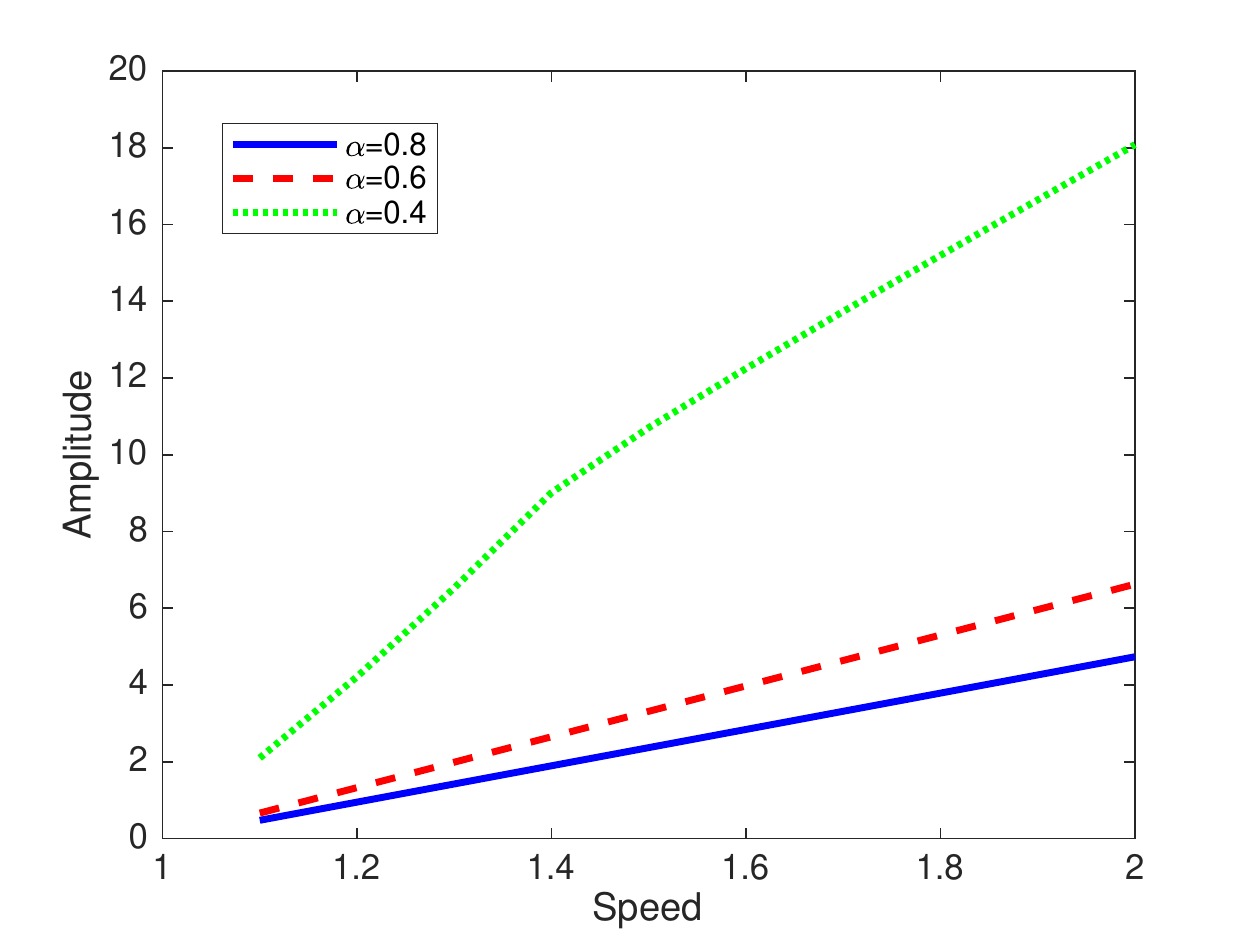}
 \end{minipage}
  \caption{Speed and amplitude relation for various values of $p$ and the fixed $\alpha=0.8$ (left panel) and for various values of $\alpha$ and fixed nonlinearity $p=1$ (right panel)}
  \label{ampspe}
\end{figure}

\subsection{The Fourier-pseudospectral method  for gfBBM Equation}

\par
To investigate the time evolution of the solutions we first show that the numerical scheme captures the exact solution \eqref{tws} for $\alpha=1$  well enough.   We use the initial data \eqref{tws} with $c=1.1$, $\alpha=1$, $p=1$. The problem is solved in the space interval $-2048 \leq x \leq 2048$ up to $T=20$. We set the  number of grid points as $N=2^{18}, M=4000$.
Figure $6$  illustrates variation of change in the conserved quantity  $I_1$ with time and shows that it is  preserved by the numerical scheme. Here we note that as the conserved quantity $I_0$ is linear it is automatically preserved by the numerical scheme \cite{gear}.


\begin{figure}[ht]
\begin{center}
   \includegraphics[width=3in]{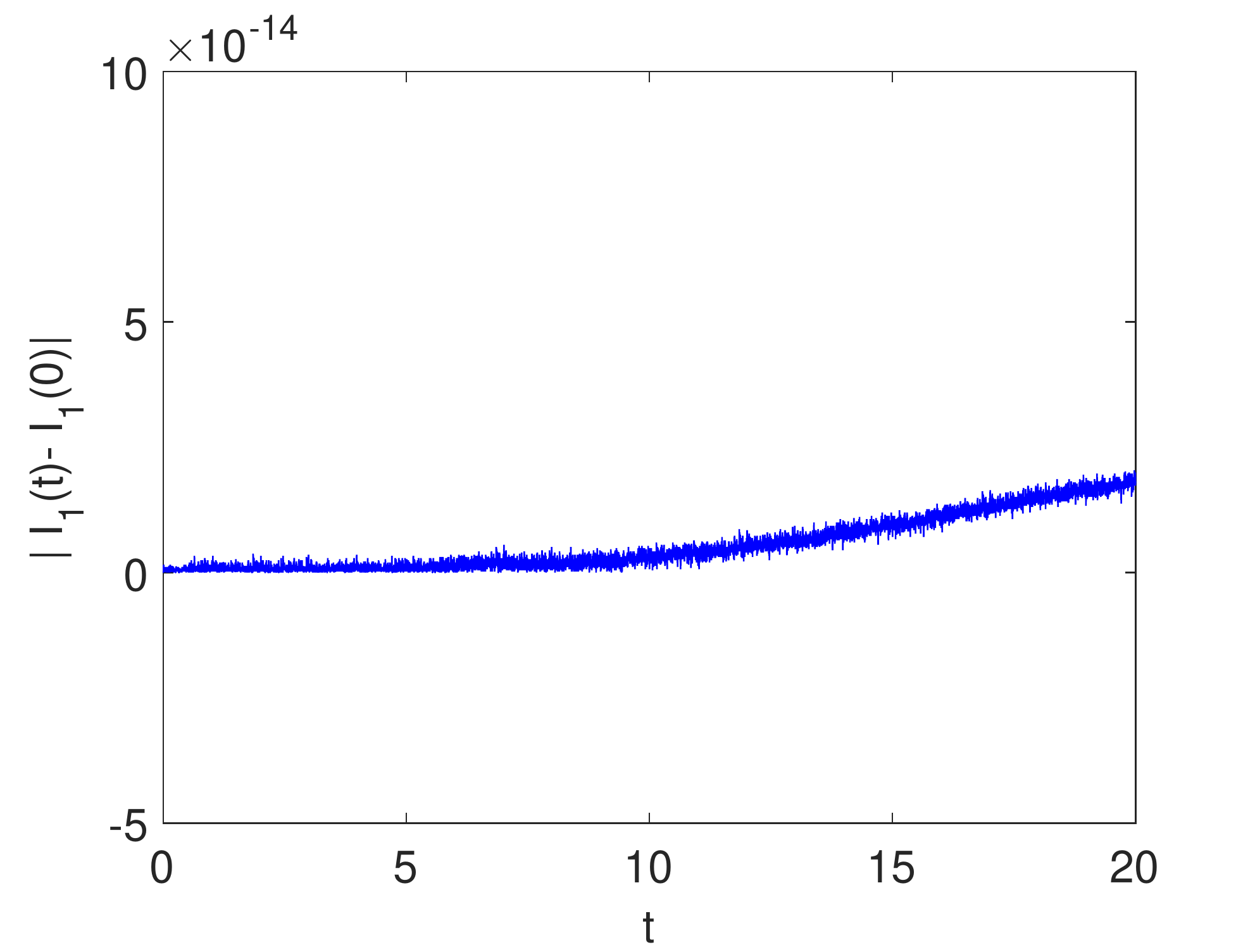}
\end{center}

\vspace*{-20pt}
  \caption{Variation of the change in the   conserved quantity $I_1$ with time}
  \label{name}
\end{figure}

Figure $7$ shows the wave profile  calculated by the Fourier pseudo-spectral method with time step $\Delta t=0.005$ at $t=10$ and $t=20$. To observe the wave profile  more clear, we focus on the space interval
$-150 \leq x \leq 150$.  We also show  the variation of  change  for   the conserved quantity  $I_1$  in  Figure $7$. It is  seen from the figure, the proposed scheme conserves  $I_1$ very well.

\begin{figure}[h!bt]
 \begin{minipage}[t]{0.45\linewidth}
   \includegraphics[width=3in]{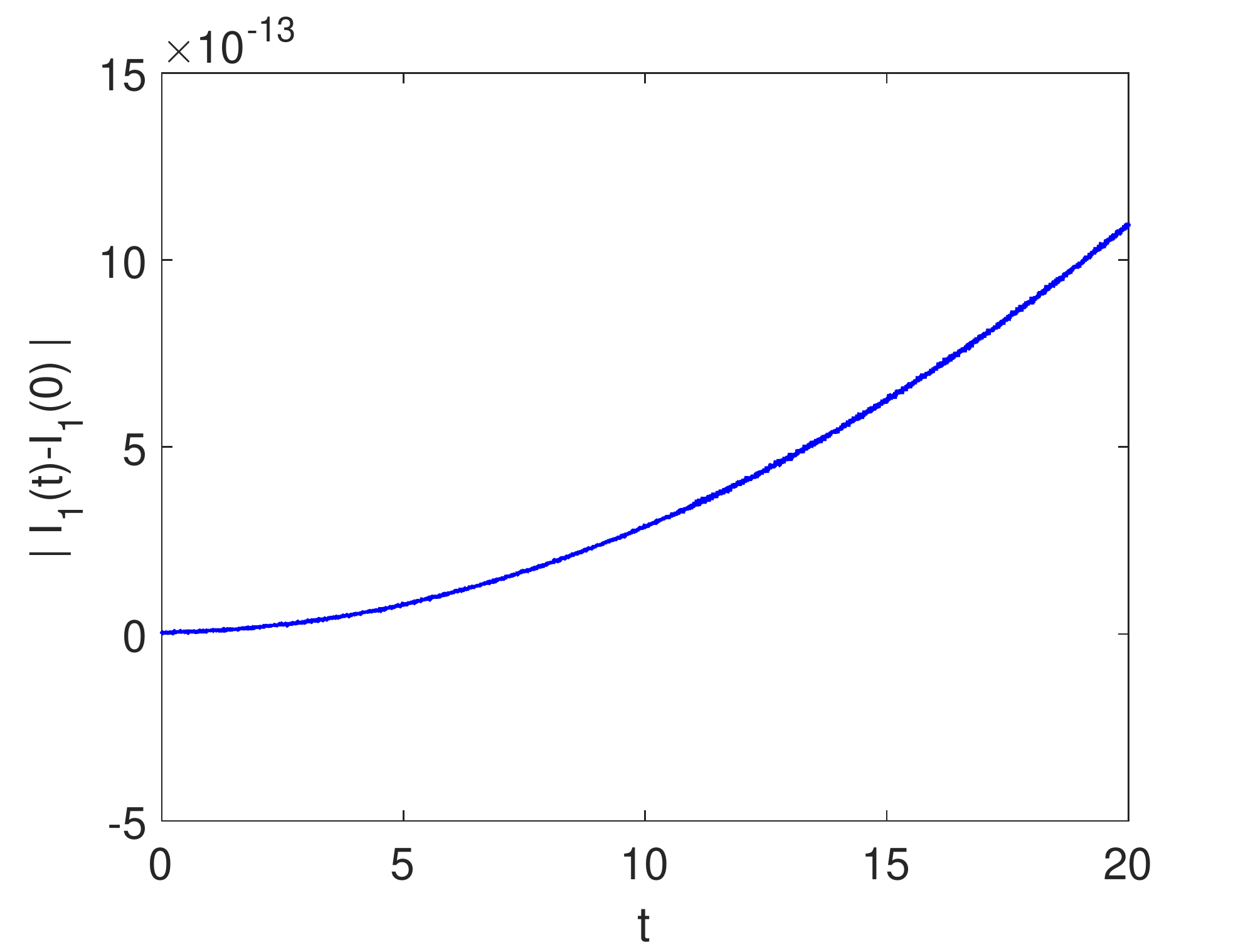}
 \end{minipage}
\hspace{30pt}
\begin{minipage}[t]{0.45\linewidth}
   \includegraphics[width=3in]{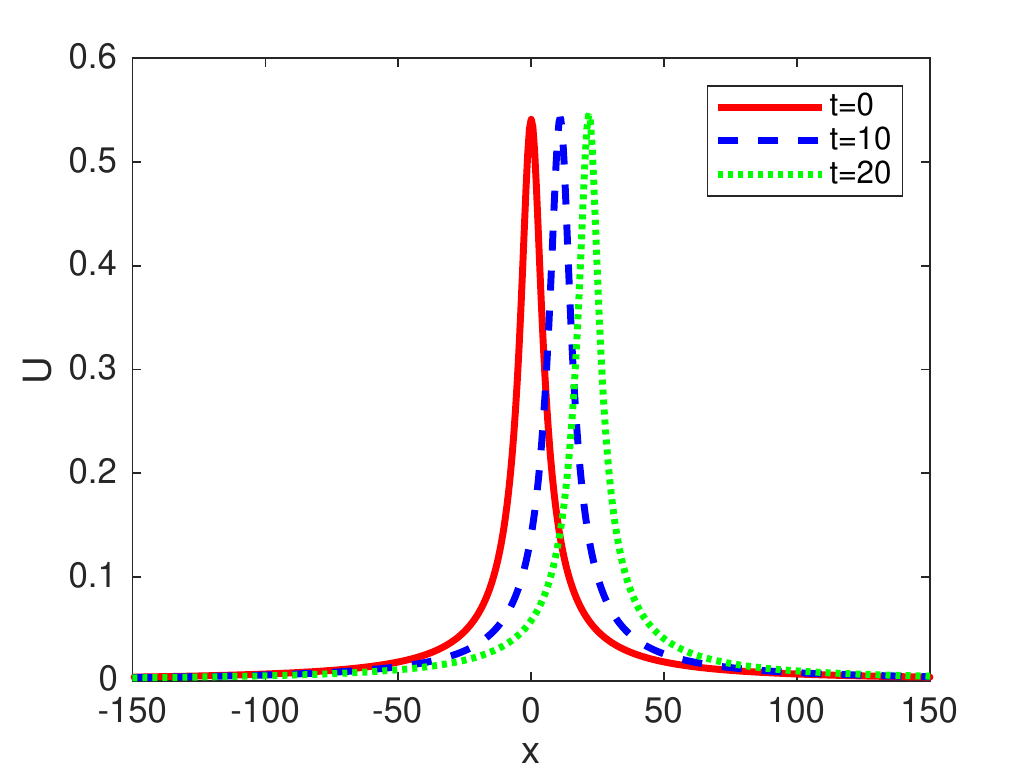}
 \end{minipage}
\caption {The change  in  the conserved quantity  $I_1$  with time (left panel) and time evolution of the wave profile ($c=1.1$, $p=1$, $\alpha=0.6$) at several times (right panel) }
\label{timeevolution}
\end{figure}

\noindent \textbf{Acknowledgements}\\
The authors would like to express  sincere gratitude to the reviewers for their constructive suggestions which helped to improve the quality of this paper.  The last author would like to thank Dr. Jiao He for helpful and fruitful communication. Goksu Oruc was supported by the Scientific and Technological Research Council of Turkey (TUBITAK) under the grant 2211. The first and the last authors were supported by Research Fund of Istanbul Technical University Project Number: 42257.
\bibliographystyle{vancouver}
\bibliography{fgBBM2}
\end{document}